\documentclass[a4paper]{amsart}

\usepackage[all,cmtip]{xy}
\usepackage{amssymb}
\usepackage{mathrsfs}
\usepackage{mathpazo}
\usepackage{setspace}
\usepackage{changepage}
\usepackage{euler}
\usepackage{leftidx}
\usepackage{xcolor}
\usepackage{hyperref}
\hypersetup{colorlinks}
\hypersetup{
    unicode=false,          
    pdftoolbar=true,        
    pdfmenubar=true,        
    pdffitwindow=false,     
    pdfstartview={FitH},    
    pdftitle={My title},    
    pdfauthor={Author},     
    pdfsubject={Subject},   
    pdfcreator={Creator},   
    pdfproducer={Producer}, 
    pdfkeywords={keyword1} {key2} {key3}, 
    pdfnewwindow=true,      
    colorlinks=true,       
    linkcolor=red,          
    citecolor=magenta,        
    filecolor=violet,      
    urlcolor=blue      
}
\usepackage[OT2,T1]{fontenc}
\DeclareSymbolFont{cyrletters}{OT2}{wncyr}{m}{n}
\DeclareMathSymbol{\Sha}{\mathalpha}{cyrletters}{"58}
\usepackage{comment}

\newtheorem{thm}{Theorem}[section]

\newtheorem{lem}[thm]{Lemma}
\newtheorem{prop}[thm]{Proposition}

\theoremstyle{definition}
\newtheorem{defn}[thm]{Definition}
\newtheorem{rmk}[thm]{Remark}
\newtheorem{notation}[thm]{Notation}

\newtheorem{ex}[thm]{Example}
\newtheorem*{comment*}{Comment}

\makeatletter
\def\blfootnote{\xdef\@thefnmark{}\@footnotetext}
\makeatother

\makeatletter
\newcommand{\holim@}[2]{%
  \vtop{\m@th\ialign{##\cr
    \hfil$#1\operator@font holim$\hfil\cr
    \noalign{\nointerlineskip\kern1.5\ex@}#2\cr
    \noalign{\nointerlineskip\kern-\ex@}\cr}}%
}
\newcommand{\varinjholim}{%
  \mathop{\mathpalette\holim@{\rightarrowfill@\textstyle}}\nmlimits@
}
\newcommand{\varprojholim}{%
  \mathop{\mathpalette\holim@{\leftarrowfill@\textstyle}}\nmlimits@
}

\newcommand{\psdlim@}[2]{%
  \vtop{\m@th\ialign{##\cr
    \hfil$#1\operator@font Psdlim$\hfil\cr
    \noalign{\nointerlineskip\kern1.5\ex@}#2\cr
    \noalign{\nointerlineskip\kern-\ex@}\cr}}%
}
\newcommand{\varinjpsdlim}{%
  \mathop{\mathpalette\psdlim@{\rightarrowfill@\textstyle}}\nmlimits@
}
\newcommand{\varprojpsdlim}{%
  \mathop{\mathpalette\psdlim@{\leftarrowfill@\textstyle}}\nmlimits@
}

\DeclareMathOperator{\im}{im}

\DeclareMathOperator{\Spec}{Spec}

\DeclareMathOperator{\cone}{cone}

\makeatother

\DeclareMathAlphabet{\mathpzc}{OT1}{pzc}{m}{it}
\numberwithin{equation}{section}

\begin{document}
\title{Arithmetic BF theory and the Cassels-Tate pairing}

\author{Jeehoon Park}
\address{Jeehoon Park: QSMS, Seoul National University, 1 Gwanak-ro, Gwanak-gu, Seoul, South Korea 08826 }
\email{jpark.math@gmail.com}

\author{Junyeong Park}
\address{Junyeong Park: Department of Mathematics Education, Chonnam National University, 77, Yongbong-ro, Buk-gu, Gwangju 61186, Korea}
\email{junyeongp@gmail.com}

\begin{abstract} 
We give a systematic treatment of the arithmetic BF theory, introduced by Carlson and Kim. We observe that the Cassels-Tate pairing can be naturally interpreted as an arithmetic BF functional.
\end{abstract}

\maketitle

\date{}

\blfootnote{2020 \textit{Mathematics Subject Classification.} Primary 11R04, 11R23, 11R34; Secondary 81T45}


\blfootnote{Key words and phrases: arithmetic gauge theory, arithmetic BF theory, Cassels-Tate pairing}

\tableofcontents

\section{Introduction}

Based on arithmetic topology (cf. \cite{Mor}), the notion of arithmetic gauge theory was introduced by Minhyong Kim \cite{kim1} as a framework for studying the deep analogies between the topology of 3-manifolds and the arithmetic of number fields. 
The first successful realization of this program was the development of arithmetic Chern-Simons theory, which has been extensively studied in recent years (cf. \cite{kim2}, \cite{CKKPY}, \cite{CKKPPY}, \cite{LP}). While Chern-Simons theory provides a rich source of arithmetic invariants, such as arithmetic linking numbers and $p$-adic $L$-functions, there remains a need for a more flexible gauge-theoretic framework capable of capturing the duality structures inherent in Selmer groups and Galois cohomology. 
The arithmetic BF theory for number fields and abelian varieties was introduced in \cite{CaKi22} to address this need and to further the philosophy of arithmetic gauge theory, which suggests that the path integral of a physical theory should be intimately related to the $L$-function or the structural invariants of the relevant arithmetic objects.

In the topological setting, BF theory is a topological quantum field theory that computes the linking of cycles (which do not require that the space-time manifold is orientable); in the arithmetic setting, it serves as a bridge between the cohomological pairings of number theory and the functional integrals of physics. The primary motivation for developing arithmetic BF theory lies in its potential to provide a "computational" or "physical" interpretation of arithmetic duality. While the initial investigations in \cite{CaKi22} provided a hint toward this philosophy, a more striking evidence appeared in the context of cyclotomic fields, where the arithmetic BF theory led to the proof of an arithmetic path integral formula for the inverse $p$-adic absolute values of Kubota-Leopoldt $p$-adic $L$-functions at roots of unity \cite{CCKKPY}. 
Subsequent work in \cite{PP25} generalized this path integral formula to elliptic curves with ordinary good reduction at $p$. These results suggest that the "action functional" of a BF-type theory is the natural home for the special values of $L$-functions, providing a topological explanation for their existence.

In the intervening period, a more systematic treatment of arithmetic Chern-Simons theory was provided in \cite{HKM}\footnote{The authors of \cite{HKM} used the terminology "arithmetic Dijkgraff-Witten theory" to indicate that the gauge group is a finite abstract group, instead of arithmetic Chern-Simons theory, which is a more broader concept (including the case that the gauge group is a $p$-adic Lie group).}. Prior to this work, foundational papers such as \cite{CKKPY} and \cite{kim2} established the initial formulations for the space of fields, the action functional, and the core components of the decomposition formula. These earlier investigations also introduced the concepts of prequantum line bundles and the partition function within the context of the quantum theory. 
However, the full quantum-theoretic structure was not yet comprehensively addressed; specifically, the rigorous construction of the quantum Hilbert space and a proper treatment of the partition function as an element within that space remained incomplete until the systematic approach in \cite{HKM} was introduced.

The goal of this article is to provide a similarly systematic treatment of abelian arithmetic BF theory for Selmer modules, adopting the style and rigor of the work in \cite{HKM}. To provide a "systematic treatment" in this context involves reconstructing the entire architecture of a physical gauge theory using purely arithmetic components. 

This process is built upon four primary pillars. 
The first pillar is the formal construction of the "space of fields," denoted as $\mathscr{F}(X_S)$, which is defined via pullbacks of cohomology groups to account for both global Galois modules and local boundary behavior at primes. 
The second pillar is the "definition of the functional via duality," where the action functional is derived from fundamental duality theorems such as Artin-Verdier or local Tate duality, ensuring the cup product pairing is globally consistent and gauge-invariant. 
The third pillar is the "categorical formulation of the partition function," which involves constructing a Hermitian line bundle $L_S$ over the space of boundary fields and defining the partition function $Z_{X_S}$ as a section; this represents the arithmetic equivalent of a physical quantum state in a Hilbert space. 
Finally, the fourth pillar is the "proof of functorial properties," which verifies that the theory obeys TQFT axioms like the "gluing formula" and "decomposition formula," allowing global invariants like the Cassels-Tate pairing to be recovered from local data.

By formalizing the space of fields and the boundary conditions in a manner consistent with the modern theory of Selmer groups, we aim to clarify the relationship between the categorical construction of the BF functional and the classical results of arithmetic duality. 
Our main observation is that the Cassels-Tate pairing, as defined for finite Galois modules in the sense of \cite{MS21} and \cite{MS24}, can be naturally interpreted as the arithmetic BF functional under certain hypotheses. 
Specifically, we show that the partition function of the arithmetic BF theory on a global field is governed by the size of the relevant Selmer groups, thereby establishing a formal link between the "quantum" state of the arithmetic manifold and the classical arithmetic invariants. This interpretation provides a unified perspective on the Cassels-Tate pairing, viewing it not merely as a pairing in cohomology, but as a physical functional defined on the space of arithmetic fields. 

Finally, we highlight the recent work in \cite{CKKPY24}, which investigates the computation of entanglement entropy within the specific framework of abelian arithmetic Chern-Simons theory. Extending these entropic investigations to the abelian arithmetic BF theory developed in this article represents a compelling direction for future research, potentially offering new insights into the information-theoretic aspects of arithmetic duality.

\textbf{Acknowledgement}:
This article is presented as a contribution to the 61st birthday conference proceedings in honor of Professor Minhyong Kim, convened from 25 to 29 Nov 2024 at ICMS, Bayes Centre, Edinburgh. We are deeply honored to participate in this tribute to Professor Kim, who fundamentally shaped the area of arithmetic quantum field theory. The authors express their heartfelt gratitude to Professor Kim for his guidance and generosity in sharing his ideas during the preparation of the paper.
We also appreciate the organizing committee for creating an exceptional workshop for academic exchange and celebration.

Jeehoon Park is very grateful to Magnus Carlson, Hee-Joong Chung, Dohyeong Kim, Minhyong Kim, and Hwajong Yoo for numerous useful discussions on arithmetic quantum field theory. He also thanks Adam Morgan for a valuable comment on the definition of the Cassels-Tate pairing.
Jeehoon Park was supported by the National Research Foundation of Korea (NRF-2021R1A2C1006696) and the National Research Foundation of Korea grant funded by the Korea government (MSIT) (No.2020R1A5A1016126). 
Junyeong Park was supported by Basic Science Research Program through the
National Research Foundation of Korea(NRF) funded by the Ministry of Education (RS-2024-00449679) and the National Research Foundation of Korea(NRF) grant funded by the Korea government (MSIT) (RS-2024-00415601).

\begin{notation}\label{notation-overall} Throughout this article, we use the following notations without further mention.
\begin{itemize}
    \item Given a category $\mathscr{C}$ and objects $s,t$ in $\mathscr{C}$, denote by $\mathscr{C}(s,t)$ the set of arrows in $\mathscr{C}$ from $s$ to $t$.
    \item Given a small category $\mathscr{C}$, denote by $\pi_0(\mathscr{C})$ its set of connected components. If $\mathscr{C}$ is a groupoid, then $\pi_0(\mathscr{C})$ becomes the set of isomorphism classes in $\mathscr{C}$.
    \item Given a map $\pi:E\rightarrow B$ in a category $\mathscr{C}$, for each subobject $U\hookrightarrow B$, denote
    \begin{align*}
        \Gamma(U,E):=\{\sigma\in\mathscr{C}(U,E)\ |\ \textrm{$\pi\circ\sigma$ is the given monomorphism $U\hookrightarrow B$}\}
    \end{align*}
    If $\mathscr{C}$ admits pullbacks, then $\Gamma(U,E)\cong\Gamma(U,E|_U)$ where $\pi|_U:E|_U\rightarrow U$ is the pullback of $\pi$ along the monomorphism $U\hookrightarrow B$.
    \item Let $\mathcal{G}$ be a commutative group scheme over a base scheme. Given an affine scheme $\Spec R$ over the same base scheme, we denote
    \begin{align*}
        H^i(R,\mathcal{G}):=H^i(\Spec R,\mathcal{G}_R),\quad i\in\mathbb{Z}
    \end{align*}
    the sheaf cohomology group for any Grothendieck topology we use.
    \item Given a field $\Bbbk$, a choice of its separable closure is denoted by $\Bbbk^\mathrm{sep}/\Bbbk$. Then denote $G_\Bbbk:=\mathrm{Gal}(\Bbbk^\mathrm{sep}/\Bbbk)$ its absolute Galois group.
    \item Given a local field $\Bbbk$, the maximal unramified extension in $\Bbbk^\mathrm{sep}/\Bbbk$ is denoted by $\Bbbk^\mathrm{nr}/\Bbbk$.
    \item Given a finite discrete $G_\Bbbk$-module $N$, denote $N^\vee:=\mathrm{Ab}(N,\Bbbk^{\mathrm{sep},\times})$.
    \item  Given a topological group $G$ and a topological $G$-module $A$, denote by $C^\bullet(G,A)$ the inhomogeneous continuous cochain complex. If $G$ is profinite and $A$ is discrete, then $C^\bullet(G,A)$ computes the group cohomologies $H^i(G,A)$. 
    See \cite[\href{https://stacks.math.columbia.edu/tag/0DVG}{Tag 0DVG}]{Stacks} for example.
    \item Given a topological group $G$ and topological $G$-modules $A,B$, we have a bilinear map
    \begin{align*}
        \xymatrix{\cup:C^p(G,A)\times C^q(G,B) \ar[r] & C^{p+q}(G,A\otimes_\mathbb{Z}B)}
    \end{align*}
    defined for $\alpha\in C^p(G,A)$ and $\beta\in C^q(G,B)$ by
    \begin{align*}
        (\alpha\cup\beta)(g_1,\cdots,g_{p+q}):=\alpha(g_1,\cdots,g_p)\otimes g_1\cdots g_p\beta(g_{p+1},\cdots,g_{p+q}).
    \end{align*}
    One can directly verify that this is a derivation on the cochain, i.e., satisfies the graded Leibniz rule:
    \begin{align*}
        d(\alpha\cup\beta)=d\alpha\cup\beta+(-1)^p\alpha\cup d\beta.
    \end{align*}
    Alternatively, one can use the inhomogeneous model (cf. \cite[I.4]{NSW08}) and the comparison isomorphism (cf. \cite[p.14]{NSW08}) to get the definition and the desired property.
    \item Given a topological group $G$ and a map $A\otimes_\mathbb{Z}B\rightarrow M$ of topological $G$-modules, the composite
    \begin{align*}
        \xymatrix{\cup:C^p(G,A)\times C^q(G,B) \ar[r]^-\cup & C^{p+q}(G,A\otimes_\mathbb{Z}B) \ar[r] & C^{p+q}(G,M)}
    \end{align*}
    will be called the \emph{cup product}. By the above observation, this is a derivation on the cochain.
\end{itemize}
\end{notation}

\begin{notation}\label{notation-Globalfield} Let $K$ be a global field.
\begin{itemize}
    \item If $K$ is of characteristic $0$, then we assume that $K$ is totally imaginary.
    \item Let $X$ be a scheme defined as follows, depending on the characteristic of $K$.
    \begin{itemize}
        \item If $K$ is of characteristic $0$, then denote by $\mathcal{O}_K\subseteq K$ its ring of integers and by $X:=\Spec\mathcal{O}_K$.
        \item If $K$ is of characteristic $p>0$, then $X$ will be a smooth proper curve over a finite field whose function field is $K$.
    \end{itemize}
    \item Given an open subset $U\subseteq X$, denote by $U^\mathrm{cl}\subseteq U$ its subset of closed points.
    \item Let $U\subseteq X$ be an open subset. For $u\in U^\mathrm{cl}$, we use the following notation.
    \begin{itemize}
        \item Denote $K_u$ the corresponding completion and $\mathcal{O}_u\subseteq K_u$ its valuation ring so that $\mathcal{O}_u\cong\widehat{\mathcal{O}}_{U,u}$.
        \item Denote by $\mathcal{O}_u^\mathrm{sep}\subseteq K_u^\mathrm{sep}$ and $\mathcal{O}_u^\mathrm{nr}\subseteq K_u^\mathrm{nr}$ the valuation rings.
        \item Denote by $G_{K,u}:=G_{K_u}$ and $I_{K,u}\subseteq G_{K,u}$ its inertia subgroup so that
        \begin{align*}
            G_{K,u}^\mathrm{nr}:=G_{K,u}/I_{K,u}\cong G_{\kappa(u)}\cong\widehat{\mathbb{Z}}\rlap{\ }
        \end{align*}
        where $\kappa(u)$ is the residue field at $u$.
    \end{itemize}
    Note that if we choose an embedding $K^\mathrm{sep}\hookrightarrow K_u^\mathrm{sep}$, then we get an embedding $G_{K,u}\hookrightarrow G_K$ whose image fixes $u\in U^\mathrm{cl}$.
    \item Let $U\subseteq X$ be an open subset. We say that a commutative group scheme over $U$ is \emph{nice} if it is finite, flat, and of order relatively prime to the characteristic of $K$.
\end{itemize}
\end{notation}

\section{Preliminaries on cup products}

\begin{rmk} Let $U\subseteq X$ be an open subset. By \cite[Corollary 11.31]{Mil17}, any nice commutative group scheme $\mathcal{G}$ over $U$ is generically \'etale, or equivalently, $\mathcal{G}_K$ is \'etale over $K$.
\end{rmk}

We define the compactly supported cohomology groups $H_{\mathrm{fppf},c}^r(U,-)$ on an open subset $U\subseteq X$ to be the cohomology of
\begin{align}\label{defn-Hc}
    \mathbb{R}\Gamma_\mathrm{fppf,c}(U,-):=\cone\left(\xymatrix{\mathbb{R}\Gamma_\mathrm{fppf}\left(U,-\right) \ar[r] & \displaystyle\bigoplus_{x\in X\setminus U}\mathbb{R}\Gamma_\mathrm{fppf}(K_x,-)}\right)[-1]
\end{align}
By \cite[Remark III.0.6 (b)]{Mil06} and \cite{DH}, we may use results in \cite[Chapter III]{Mil06}. We let
\begin{align}\label{defn-Hc-im}
    D_\mathrm{fppf}^r(U,-):=\im\left(\xymatrix{H_\mathrm{fppf,c}^r(U,-) \ar[r] & H_\mathrm{fppf}^r(U,-)}\right)\rlap{\ .}
\end{align}

\begin{rmk} If $U=X$, then $D_\mathrm{fppf}^r(X,-)=H_\mathrm{fppf}^r(X,-)$ because the induced map
\begin{align*}
    \xymatrix{H_\mathrm{fppf,c}^r(X,-) \ar[r] & H_\mathrm{fppf}^r(X,-)}
\end{align*}
becomes an isomorphism.
\end{rmk}

On the other hand, consider the canonical isomorphism (see \cite[III.3 and III.8]{Mil06} or \cite[section 4 and section 5]{DH})
\begin{align}\label{Global-H3-inv}
    \xymatrix{\displaystyle\int_U:H_\mathrm{fppf,c}^3(U,\mathbb{G}_m) \ar[r]^-\sim & \mathbb{Q}/\mathbb{Z}\rlap{\ .}}
\end{align}
By \cite[Corollary III.3.2 and Theorem III.8.2]{Mil06} or \cite[Theorem 1.1]{DH}, if $\mathcal{M}$ is a nice commutative group scheme over $U$ with the Cartier dual $\mathcal{M}^\vee$, then
\begin{align}\label{AVduality}
    \xymatrix{H_\mathrm{fppf}^r(U,\mathcal{M}^\vee)\times H_\mathrm{fppf,c}^{3-r}(U,\mathcal{M}) \ar[r] & \mathbb{Q}/\mathbb{Z} & (\alpha,\beta) \ar@{|->}[r] & \displaystyle\int_U\alpha\cup\beta}
\end{align}
is a perfect pairing for $r\in\{0,1,2,3\}$.

\begin{lem} The pairing \eqref{AVduality} induces a perfect pairing
\begin{align*}
    \xymatrix{D_\mathrm{fppf}^r(U,\mathcal{M}^\vee)\times D_\mathrm{fppf}^{3-r}(U,\mathcal{M}) \ar[r] & \mathbb{Q}/\mathbb{Z}\rlap{\ .}}
\end{align*}
such that the following diagram is commutative:
\begin{align*}
\xymatrixcolsep{1pc}\xymatrix{
H_\mathrm{fppf}^r(U,\mathcal{M}^\vee)  & \times & H_\mathrm{fppf,c}^{3-r}(U,\mathcal{M}) \ar@{->>}[d] \ar[rr] & & \mathbb{Q}/\mathbb{Z} \ar@{=}[d] \\
D_\mathrm{fppf}^r(U,\mathcal{M}^\vee) \ar@{^(->}[u] & \times & D_\mathrm{fppf}^{3-r}(U,\mathcal{M}) \ar[rr] & & \mathbb{Q}/\mathbb{Z}\rlap{\ .}
}
\end{align*}
\end{lem}
\begin{proof} This is an fppf cohomology analogue of \cite[Corollary II.3.4]{Mil06}. See \cite[Theorem 2.16]{Kes17} for a proof.
\end{proof}

To define the arithmetic BF-functional, we need an exact sequence of nice commutative group schemes over an open subset $Y\subseteq X$ as an input:
\begin{align}\label{BF-input}
    \xymatrix{0 \ar[r] & \mathcal{M}_1 \ar[r] & \mathcal{M} \ar[r] & \mathcal{M}_2 \ar[r] & 0\rlap{\ .}}
\end{align}
\begin{ex} Let $A$ be an abelian variety over $K$. Take the N\'eron model $\mathcal{A}$ over $X$. This is a commutative quasi-projective group scheme over $X$. However, $\mathcal{A}$ is not necessarily proper over $X$. If $Y\subseteq X$ is the open subset where $\mathcal{A}$ admits good reduction, then $\mathcal{A}|_Y$ becomes an abelian scheme over $Y$. For each $n\in\mathbb{Z}$ relatively prime to the characteristic of $K$, the torsion subgroups $\mathcal{A}|_Y[n^r]\subseteq\mathcal{A}|_Y$ are finite and flat over $Y$. Then the exact sequence
\begin{align*}
    \xymatrix{0 \ar[r] & \mathcal{A}|_Y[n] \ar[r] & \mathcal{A}|_Y[n^2] \ar[r]^-n & \mathcal{A}|_Y[n] \ar[r] & 0}
\end{align*}
is an example of \eqref{BF-input}. Moreover, $\mathcal{A}|_Y[n]$ is finite \'etale over the affine open subset $Y[1/n]\subseteq Y$ defined by inverting $n$.
\end{ex}
\begin{ex} Let $E$ be an elliptic curve over $K$ which admits a minimal Weierstrass model $\mathcal{E}$ over $X$. In general, $\mathcal{E}$ is not a group scheme over $X$ and only its smooth locus $\mathcal{E}^\mathrm{sm}\subseteq\mathcal{E}$ is a group scheme over $X$. If $Y\subseteq X$ is the open subset where $\mathcal{E}$ admits good reduction, then $\mathcal{E}|_Y$ becomes an abelian scheme over $Y$. For each $n\in\mathbb{Z}$ relatively prime to the characteristic of $K$, the torsion subgroups $\mathcal{E}|_Y[n^r]\subseteq\mathcal{E}|_Y$ are finite flat over $Y$. Then the exact sequence
\begin{align*}
    \xymatrix{0 \ar[r] & \mathcal{E}|_Y[n] \ar[r] & \mathcal{E}|_Y[n^2] \ar[r]^-n & \mathcal{E}|_Y[n] \ar[r] & 0}
\end{align*}
is an example of \eqref{BF-input}. Moreover, $\mathcal{E}|_Y[n]$ is finite \'etale over the affine open subset $Y[1/n]\subseteq Y$ defined by inverting $n$.
\end{ex}
Applying \eqref{defn-Hc} term by term, using \cite[\href{https://stacks.math.columbia.edu/tag/05R0}{Tag 05R0}]{Stacks}, and taking cohomology groups, we get an anticommutative square:
\begin{align*}
    \xymatrix{
    H_\mathrm{fppf,c}^1(Y,\mathcal{M}_2) \ar[d] \ar[r]^-\delta & H_\mathrm{fppf,c}^2(Y,\mathcal{M}_1) \ar[d] \\
    H_\mathrm{fppf}^1(Y,\mathcal{M}_2) \ar[r]_-\delta & H_\mathrm{fppf}^2(Y,\mathcal{M}_1)
    }
\end{align*}
where each $\delta$ denotes the connecting map. Pasting this with the cup product diagram, we get the following commutative square:
\begin{align*}
    \xymatrix{
    D_\mathrm{fppf}^1(Y,\mathcal{M}_1^\vee)\times H_\mathrm{fppf,c}^1(Y,\mathcal{M}_2) \ar@{->>}[d] \ar[r] & \mathbb{Q}/\mathbb{Z} \ar[d]^-{-1} & (\alpha,\beta) \ar@{|->}[r] & \displaystyle\int_Y\alpha\cup\delta\beta \\
    D_\mathrm{fppf}^1(Y,\mathcal{M}_1^\vee)\times D_\mathrm{fppf}^1(Y,\mathcal{M}_2) \ar[r] & \mathbb{Q}/\mathbb{Z} & &
    }
\end{align*}
Denote the bottom row as follows:
\begin{align}\label{D1D1-pairing}
    \xymatrixcolsep{1.5pc}\xymatrix{D_\mathrm{fppf}^1(Y,\mathcal{M}_1^\vee)\times D_\mathrm{fppf}^1(Y,\mathcal{M}_2) \ar[r] & \mathbb{Q}/\mathbb{Z} & (\alpha,\beta) \ar@{|->}[r] & \displaystyle\int_Y\alpha\cup\delta\beta\rlap{\ .}}
\end{align}

Note that we have \'etale analogues of \eqref{defn-Hc}, \eqref{defn-Hc-im}, \eqref{Global-H3-inv}, and a canonical comparison map coming from \eqref{et-fppf-map}, \eqref{et,c-fppf,c-map}. Since the canonical map is compatible with connecting maps and cup products, \eqref{BF-input} gives a commutative square:
\begin{align*}
    \xymatrix{
    D_\mathrm{\acute{e}t}^1(Y,\mathcal{M}_1^\vee)\times D_\mathrm{\acute{e}t}^1(Y,\mathcal{M}_2) \ar[d] \ar[r]^-{\eqref{D1D1-pairing}^\mathrm{\acute{e}t}} & \mathbb{Q}/\mathbb{Z} \ar@{=}[d] \\
    D_\mathrm{fppf}^1(Y,\mathcal{M}_1^\vee)\times D_\mathrm{fppf}^1(Y,\mathcal{M}_2) \ar[r]_-{\eqref{D1D1-pairing}} & \mathbb{Q}/\mathbb{Z}
    }
\end{align*}
where $\eqref{D1D1-pairing}^\mathrm{\acute{e}t}$ is an analogously defined map for the \'etale cohomology. By Remark \ref{et-fppf-small}, the left column becomes an isomorphism whenever it is restricted to a sufficiently small open subset $U\subseteq Y$.

\section{Arithmetic BF theory for local fields}\label{Localthy}

This section develops the local components of arithmetic BF theory at each closed point $x \in X^{\mathrm{cl}}$, treating the local fields $K_x$ as the 2-dimensional building blocks of a global arithmetic 3-manifold. We reconstruct the gauge-theoretic architecture through the construction of $\mathbb{Q}/\mathbb{Z}$-torsors over local cohomology groups and the isolation of unramified configurations. This categorical framework ensures that local boundary data can be rigorously glued to recover global invariants like the Cassels-Tate pairing. We begin our construction with the short exact sequence of finite discrete $G_K$-modules provided below, which serves as the primary input for the theory and ensures that the local theory remains grounded in the structure of the Galois modules;
we work with an exact sequence of finite discrete $G_K$-modules:
\begin{align}\label{BF-input-GK}
    \xymatrix{0 \ar[r] & M_1 \ar[r]^-\iota & M \ar[r]^-\pi & M_2 \ar[r] & 0\rlap{\ .}}
\end{align}
\begin{ex}\label{fppf-Gal} Evaluating \eqref{BF-input} at $K^\mathrm{sep}$ and using Remark \ref{et-grp}, we get an exact sequence of finite discrete $G_K$-modules:
\begin{align*}
    \xymatrix{0 \ar[r] & \mathcal{M}_1(K^\mathrm{sep}) \ar[r] & \mathcal{M}(K^\mathrm{sep}) \ar[r] & \mathcal{M}_2(K^\mathrm{sep}) \ar[r] & 0\rlap{\ .}}
\end{align*}
This is the example we have in mind. We will see that our arithmetic BF theory for local fields (2-dimensional objects) depends only on the Galois modules.
\end{ex}

In the rest of this section, regarding this exact sequence as input data, we will construct $\mathbb{Q}/\mathbb{Z}$-torsors of the form $\varpi_S:\mathscr{L}_S\rightarrow\mathscr{F}_S$ and $\varpi_S^\mathrm{nr}:\mathscr{L}_S^\mathrm{nr}\rightarrow\mathscr{F}_S^\mathrm{nr}$ for each finite subset $S\subseteq X^\mathrm{cl}$. Furthermore, we construct a section $\mathrm{BF}_S^\mathrm{nr}:\mathscr{F}_S^\mathrm{nr}\rightarrow\mathscr{L}_S^\mathrm{nr}$ of $\varpi_S^\mathrm{nr}$, which serves as the local version of the global BF functional to be developed in Section \ref{Global-BF}. This construction effectively mirrors the topological principle that boundary dynamics are determined by local cohomological data. Finally, the section $\mathrm{BF}_S^\mathrm{nr}$ will play a critical role in the proof of the decomposition formula presented in Theorem \ref{Decomp-formula}.

\subsection{Constructing $\mathbb{Q}/\mathbb{Z}$-torsors}\label{Localthy-torsor}

For each $x\in X^\mathrm{cl}$, we denote
\begin{align*}
    \mathscr{F}_x:=H^1(G_{K,x},M_1^\vee)\times H^1(G_{K,x},M_2)\rlap{\ .}
\end{align*}
Note that this is finite discrete by \cite[Theorem 7.1.8]{NSW08}. Using the group cocycles, we will construct a $\mathbb{Q}/\mathbb{Z}$-torsor $\varpi_x:\mathscr{L}_x\rightarrow\mathscr{F}_x$. For the local analysis, we use the canonical isomorphism (see \cite[I.A]{Mil06})
\begin{align}\label{Local-H2-inv}
    \xymatrix{\mathrm{inv}_x:H^2(G_{K,x},K_x^{\mathrm{sep},\times}) \ar[r]^-\sim & \mathbb{Q}/\mathbb{Z}}
\end{align}
which is compatible with the above identifications.

\begin{defn} Given \eqref{BF-input-GK}, define for each $x\in X^\mathrm{cl}$ a category $\mathscr{A}_x$ to be the action groupoid associated to the homomorphism
\begin{align*}
    \xymatrix{d:C^0(G_{K,x},M_1^\vee)\times C^0(G_{K,x},M_2) \ar[r] & Z^1(G_{K,x},M_1^\vee)\times Z^1(G_{K,x},M_2)\rlap{\ .}}
\end{align*}
Explicitly, the category $\mathscr{A}_x$ is described as follows.
\begin{itemize}
    \item $\mathrm{ob}(\mathscr{A}_x):=Z^1(G_{K,x},M_1^\vee)\times Z^1(G_{K,x},M_2)$.
    \item Given $(\alpha_1,\alpha_2),(\beta_1,\beta_2)\in\mathscr{A}_x$, define
    \begin{align*}
        &\quad\mathscr{A}_x((\alpha_1,\alpha_2),(\beta_1,\beta_2))\\
        &:=\left\{(g_1,g_2)\in C^0(G_{K,x},M_1^\vee)\times C^0(G_{K,x},M_2)\ \middle|\ \beta_i=\alpha_i+dg_i\right\}\rlap{\ .}
    \end{align*}
    \item Composition is defined via the addition of arrows.
\end{itemize}
\end{defn}

\begin{rmk}\label{Rmk-Cx} Let $x\in X^\mathrm{cl}$.
\begin{enumerate}
    \item For each $(\alpha_1,\alpha_2)\in\mathscr{A}_x$, we have
    \begin{align*}
        \mathrm{End}_{\mathscr{A}_x}(\alpha_1,\alpha_2)=Z^0(G_{K,x},M_1^\vee)\times Z^0(G_{K,x},M_2)\rlap{\ .}
    \end{align*}
    \item $\pi_0(\mathscr{A}_x)\cong H^1(G_{K,x},M_1^\vee)\times H^1(G_{K,x},M_2)=\mathscr{F}_x$.
\end{enumerate}
\end{rmk}

\begin{defn} Given \eqref{BF-input-GK}, define a category $\mathscr{R}$ as follows.
\begin{itemize}
    \item $\mathrm{ob}(\mathscr{R}):=\left\{\sigma\in\mathrm{Set}(M_2,M)\ \middle|\ \pi\circ\sigma=\mathrm{Id}_{M_2}\right\}$.
    \item Given $\sigma,\tau\in\mathscr{R}$, define
    \begin{align*}
        \mathscr{R}(\sigma,\tau):=\{f\in\mathrm{Set}(M_2,M_1)\ |\ \tau-\sigma=\iota\circ f\}\rlap{\ .}
    \end{align*}
    \item Composition is defined via the addition of arrows.
\end{itemize}
\end{defn}

The category $\mathscr{R}$ is designed to manage the ambiguity inherent in lifting local boundary data to the global (physical) field configurations.

\begin{rmk} The $\iota$ in \eqref{BF-input-GK} is injective. Hence for each pair $\sigma,\tau\in\mathscr{R}$ there is a unique arrow $\iota^{-1}\circ(\tau-\sigma)\in\mathscr{R}(\sigma,\tau)$. Consequently, $\mathscr{R}$ is a connected groupoid, and every object of $\mathscr{R}$ has the trivial automorphism group.
\end{rmk}

Combining \cite[Corollary 7.2.2]{NSW08} and \cite[Proposition 1.2.5]{NSW08}, we get
\begin{align}\label{local-H3Gm-vanishing}
    H^3(G_{K,x},K_x^{\mathrm{sep},\times})=0\rlap{\ .}
\end{align}
for each $x\in X^\mathrm{cl}$. Hence we have an exact sequence
\begin{align*}
    \xymatrixcolsep{1.5pc}\xymatrix{0 \ar[r] & Z^2(G_{K,x},K_x^{\mathrm{sep},\times}) \ar[r] & C^2(G_{K,x},K_x^{\mathrm{sep},\times}) \ar[r]^-d & Z^3(G_{K,x},K_x^{\mathrm{sep},\times}) \ar[r] & 0\rlap{\ .}}
\end{align*}
Taking mod $B^2(G_{K,x},K_x^{\mathrm{sep},\times})$, we get an exact sequence
\begin{align}\label{Local-fiberseq}
    \xymatrixcolsep{1.4pc}\xymatrix{0 \ar[r] & H^2(G_{K,x},K_x^{\mathrm{sep},\times}) \ar[r] & \displaystyle\frac{C^2(G_{K,x},K_x^{\mathrm{sep},\times})}{B^2(G_{K,x},K_x^{\mathrm{sep},\times})} \ar[r]^-d & Z^3(G_{K,x},K_x^{\mathrm{sep},\times}) \ar[r] & 0\rlap{\ .}}
\end{align}

\begin{defn} Let $x\in X^\mathrm{cl}$.
\begin{itemize}
    \item Denote by $\mathscr{E}_x$ the action groupoid under the translation action on the middle term of \eqref{Local-fiberseq}:
    \begin{align*}
        \xymatrixcolsep{1.2pc}\xymatrix{\displaystyle+:\frac{C^2(G_{K,x},K_x^{\mathrm{sep},\times})}{B^2(G_{K,x},K_x^{\mathrm{sep},\times})}\times\frac{C^2(G_{K,x},K_x^{\mathrm{sep},\times})}{B^2(G_{K,x},K_x^{\mathrm{sep},\times})} \ar[r] & \displaystyle\frac{C^2(G_{K,x},K_x^{\mathrm{sep},\times})}{B^2(G_{K,x},K_x^{\mathrm{sep},\times})}\rlap{\ .}}
    \end{align*}
    \item Denote $\mathscr{B}_x$ the action groupoid under the translation action on the last term of \eqref{Local-fiberseq}:
    \begin{align*}
        \xymatrix{+:Z^3(G_{K,x},K_x^{\mathrm{sep},\times})\times Z^3(G_{K,x},K_x^{\mathrm{sep},\times}) \ar[r] & Z^3(G_{K,x},K_x^{\mathrm{sep},\times})}
    \end{align*}
\end{itemize}
\end{defn}

The category $\mathscr{A}_x$ represents the local (physical) fields and their gauge transformations, while $\mathscr{E}_x$ and $\mathscr{B}_x$ handle the secondary cohomology classes and their differentials.

\begin{rmk}\label{Rmk-ExBx} Let $x\in X^\mathrm{cl}$.
\begin{enumerate}
    \item $\mathscr{E}_x$ and $\mathscr{B}_x$ are connected groupoids such that every object has the trivial automorphism group.
    \item The $d$ in \eqref{Local-fiberseq} induces a functor $\mathbf{d}:\mathscr{E}_x\rightarrow\mathscr{B}_x$. This functor is surjective on objects and arrows because $d$ is surjective.
\end{enumerate}
\end{rmk}

\begin{lem}\label{cup-functor} Given \eqref{BF-input-GK}, for each $x\in X^\mathrm{cl}$ we can define a functor $\Theta_x:\mathscr{R}\times\mathscr{A}_x\rightarrow\mathrm{Set}$ as follows.
\begin{itemize}
    \item  For objects, we take
    \begin{align*}
        \xymatrix{\mathscr{R}\times\mathscr{A}_x \ar[r] & \mathrm{Set} & (\sigma,\alpha_1,\alpha_2) \ar@{|->}[r] & \mathrm{ob}\left(\mathbf{d}^{-1}(\alpha_1\cup d(\sigma\circ\alpha_2))\right)}
    \end{align*}
    where $\mathbf{d}:\mathscr{E}_x\rightarrow\mathscr{B}_x$ is as in Remark \ref{Rmk-ExBx}, and $\{\alpha_1\cup d(\sigma\circ\alpha_2)\}\subseteq\mathscr{B}_x$ is regarded as a discrete category.
    \item Given $\sigma,\tau\in\mathscr{R}$ and an arrow $(g_1,g_2):(\alpha_1,\alpha_2)\rightarrow(\beta_1,\beta_2)$ in $\mathscr{A}_x$, we take
    \begin{align*}
    &\quad\Theta_x(\tau-\sigma,g_1,g_2)\\
    &:=\alpha_1\cup((\tau-\sigma)\circ\alpha_2)+g_1\cup d(\tau\circ\alpha_2)+\beta_1\cup(\tau\circ dg_2-d(\tau\circ g_2))
    \end{align*}
    acting by translation on the middle term of \eqref{Local-fiberseq}.
\end{itemize}
\end{lem}
\begin{proof} The functoriality of $\Theta_x$ is equivalent to the biadditivity of $\Theta_x$ on arrows. Note that composition on $\mathscr{E}_x$ is defined via the addition of arrows and we have
\begin{align*}
    \Theta_x(\tau-\sigma,g_1,g_2)=\Theta_x(\tau-\sigma,0_{(\alpha_1,\alpha_2)})+\Theta_x(0_\tau,g_1,g_2)
\end{align*}
with
\begin{align*}
    \Theta_x(\tau-\sigma,0_{(\alpha_1,\alpha_2)})=\alpha_1\cup((\tau-\sigma)\circ\alpha_2)
\end{align*}
\begin{align*}
    \Theta_x(0_\tau,g_1,g_2)=g_1\cup d(\tau\circ\alpha_2)+\beta_1\cup(\tau\circ dg_2-d(\tau\circ g_2))\rlap{\ .}
\end{align*}
Hence it suffices to verify the additivity of $\Theta_x$ on each component $\mathscr{R}$ and $\mathscr{A}_x$. The additivity on $\mathscr{R}$ follows immediately from the definition.  Given another arrow $(h_1,h_2):(\beta_1,\beta_2)\rightarrow(\gamma_1,\gamma_2)$, we have
\begin{align*}
    &\quad\Theta_x(0_\sigma,g_1+h_1,g_2+h_2)-\Theta_x(0_\sigma,g_1,g_2)-\Theta_x(0_\sigma,h_1,h_2)\\
    &=(g_1+h_1)\cup d(\sigma\circ\alpha_2)+\gamma_1\cup(\sigma\circ d(g_2+h_2)-d(\sigma\circ(g_2+h_2)))\\
    &\quad-g_1\cup d(\sigma\circ\alpha_2)-\beta_1\cup(\sigma\circ dg_2-d(\sigma\circ g_2))\\
    &\quad\quad-h_1\cup d(\sigma\circ\beta_2)-\gamma_1\cup(\sigma\circ dh_2-d(\sigma\circ h_2))\\
    &=dh_1\cup(\sigma\circ dg_1-d(\sigma\circ g_1))-h_1\cup d(\sigma\circ dg_2)\\
    &=d(h_1\cup(\sigma\circ dg_2-d(\sigma\circ g_1)))\rlap{\ .}
\end{align*}
Since we are working modulo $B^2(G_{K,x},K_x^{\mathrm{sep},\times})$, the desired additivity follows.
\end{proof}

We may regard $\Theta_x$ as an action of $\mathscr{R}\times\mathscr{A}_x$ on $\mathscr{E}_x$ in the following sense. We have a disjoint union decomposition:
\begin{align*}
    \mathrm{ob}(\mathscr{E}_x)=\frac{C^2(G_{K,x},K_x^{\mathrm{sep},\times})}{B^2(G_{K,x},K_x^{\mathrm{sep},\times})}=\coprod_{z\in Z^3(G_{K,x},K_x^{\mathrm{sep},\times})}d^{-1}(z)\rlap{\ .}
\end{align*}
Then $\Theta_x$ sends arrows in $\mathscr{R}\times\mathscr{A}_x$ to isomorphisms of full subcategories spanned by the corresponding summands, and these isomorphisms are natural in $\mathscr{R}\times\mathscr{A}_x$. Consequently, the equivalence relation on $\mathrm{ob}(\mathscr{E}_x)$ defined by the $\Theta_x$-action is stable under composition, and hence we have a quotient category:
\begin{align*}
    \xymatrix{\mathscr{E}_x \ar[r] & \mathscr{E}_x/\Theta_x\rlap{\ .}}
\end{align*}

On the other hand, applying $\mathbf{d}$ to the values of $\Theta_x$, we get a functor
\begin{align*}
    \xymatrixrowsep{0pc}\xymatrix{
    & & (\sigma,\alpha_1,\alpha_2) \ar@{|->}[r] & \alpha_1\cup d(\sigma\circ\alpha_2) \\
    \mathbf{d}_\ast\Theta_x:\mathscr{R}\times\mathscr{A}_x \ar[r] & \mathscr{B}_x & & \\
    & & (\tau-\sigma,g_1,g_2) \ar@{|->}[r] & {\begin{array}{c} \textrm{translation by} \\ d(\Theta_x(\tau-\sigma,g_1,g_2)) \end{array}}
    }
\end{align*}

Let $x\in X^\mathrm{cl}$. For each $(\sigma,\alpha_1,\alpha_2)\in\mathscr{R}\times\mathscr{A}_x$ define $\mathscr{L}_x^\sigma(\alpha_1,\alpha_2)$ to be the fiber, regarding $\{(\sigma,\alpha_1,\alpha_2)\}$ as a discrete category:
\begin{align*}
    \xymatrix{
    \mathscr{L}_x^\sigma(\alpha_1,\alpha_2) \ar[d] \ar@{^(->}[rr] & & \mathscr{E}_x \ar[d]^-{\mathbf{d}} \\
    \{(\sigma,\alpha_1,\alpha_2)\} \ar@{^(->}[r] & \mathscr{R}\times\mathscr{A}_x \ar[r]_-{\mathbf{d}_\ast\Theta_x} & \mathscr{B}_x
    }
\end{align*}
By Remark \ref{Rmk-ExBx}, $\mathscr{L}_x^\sigma(\alpha_1,\alpha_2)\subseteq\mathscr{E}_x$ is a full subcategory with
\begin{align*}
    \mathrm{ob}(\mathscr{L}_x^\sigma(\alpha_1,\alpha_2))=d^{-1}(\alpha_1\cup d(\sigma\circ\alpha_2))
\end{align*}
which is an $H^2(G_{K,x},K_x^{\mathrm{sep},\times})$-torsor because of \eqref{Local-fiberseq}. Along the isomorphism \eqref{Local-H2-inv}, $\mathscr{L}_x^\sigma(\alpha_1,\alpha_2)$ becomes a $\mathbb{Q}/\mathbb{Z}$-torsor. Moreover, by Lemma \ref{cup-functor}, given $\sigma,\tau\in\mathscr{R}$ and an arrow $(g_1,g_2):(\alpha_1,\alpha_2)\rightarrow(\beta_1,\beta_2)$ in $\mathscr{A}_x$, addition by $\Theta_x(\tau-\sigma,g_1,g_2)$ induces an isomorphism of $\mathbb{Q}/\mathbb{Z}$-torsors:
\begin{align}\label{d-fiber-isom}
    \xymatrix{\Theta_x(\tau-\sigma,g_1,g_2):\mathscr{L}_x^\sigma(\alpha_1,\alpha_2) \ar[r]^-\sim & \mathscr{L}_x^\tau(\beta_1,\beta_2)}
\end{align}
Denote $\mathscr{T}_1$ the category of $\mathbb{Q}/\mathbb{Z}$-torsors. From the above observation, we deduce that the following assignment defines a functor:
\begin{align*}
    \xymatrix{\mathscr{R}\times\mathscr{A}_x \ar[r] & \mathscr{T}_1 & (\sigma,\alpha_1,\alpha_2) \ar@{|->}[r] & \mathscr{L}_x^\sigma(\alpha_1,\alpha_2)\rlap{\ .}}
\end{align*}
On the other hand, regarding sets as discrete categories, take a pushout\footnote{It is well-known that the category $\mathrm{Cat}$ of small categories is cocomplete. For example, there is proof using the adjoint functor theorem. For elementary construction of coequalizers in $\mathrm{Cat}$, see \cite{BBP99}.}
\begin{align*}
    \xymatrix{
    \mathscr{R}\times\mathscr{A}_x \ar[d] \ar[r]^-{\mathbf{d}_\ast\Theta_x} & \mathscr{B}_x \ar[d] \\
    \pi_0(\mathscr{R}\times\mathscr{A}_x) \ar[r] & \mathscr{B}_x/(\mathbf{d}_\ast\Theta_x)
    }
\end{align*}
i.e., identify objects that can be connected by a zigzag of arrows in the image of $\mathscr{R}\times\mathscr{A}_x$. Then we have the following diagram of functors (cf. Remark \ref{Rmk-Cx}):
\begin{align*}
    \xymatrix{
    & \mathscr{E}_x/\Theta_x \ar[d]^-{\mathbf{d}} \\
    \mathscr{F}_x \ar[r]_-{\mathbf{d}_\ast\Theta_x} & \mathscr{B}_x/(\mathbf{d}_\ast\Theta_x)\rlap{\ .}
    }
\end{align*}
By the above observation, $\Theta_x$ only connects elements in various $\mathscr{L}_x^\sigma(\alpha_1,\alpha_2)$ and does not connect two different elements in $\mathscr{L}_x^\sigma(\alpha_1,\alpha_2)$. Hence the composite
\begin{align*}
    \xymatrix{\mathscr{L}_x^\sigma(\alpha_1,\alpha_2) \ar@{^(->}[r] & \mathscr{E}_x \ar[r] & \mathscr{E}_x/\Theta_x\rlap{\ .}}
\end{align*}
is still injective, and the isomorphisms \eqref{d-fiber-isom} map to the identity maps along the quotient functor $\mathscr{E}_x\rightarrow\mathscr{E}_x/\Theta_x$. Moreover, each fiber
\begin{align*}
    \xymatrix{
    \mathscr{L}_x(\rho_1,\rho_2) \ar[d] \ar@{^(->}[rr] & & \mathscr{E}_x/\Theta_x \ar[d]^-{\mathbf{d}} \\
    \{(\rho_1,\rho_2)\} \ar@{^(->}[r] & \mathscr{F}_x \ar[r]_-{\mathbf{d}_\ast\Theta_x} & \mathscr{B}_x/(\mathbf{d}_\ast\Theta_x)
    }
\end{align*}
is uniquely identified with every $\mathscr{L}_x^\sigma(\alpha_1,\alpha_2)$ such that $[\alpha_i]=\rho_i$. Since the quotient map $\mathscr{E}_x\rightarrow\mathscr{E}_x/\Theta_x$ is $\mathbb{Q}/\mathbb{Z}$-equivariant, $\mathscr{L}_x(\rho_1,\rho_2)$ is still a $\mathbb{Q}/\mathbb{Z}$-torsor. Denote
\begin{align*}
    \mathscr{L}_x:=\coprod_{(\rho_1,\rho_2)\in\mathscr{F}_x}\mathscr{L}_x(\rho_1,\rho_2)
\end{align*}
the disjoint union. Then the projection to its index set:
\begin{align*}
    \xymatrix{\displaystyle\varpi_x:\mathscr{L}_x \ar[r] & \mathscr{F}_x}
\end{align*}
becomes a $\mathbb{Q}/\mathbb{Z}$-torsor over $\mathscr{F}_x$. Since $\mathscr{F}_x$ is finite discrete by Remark \ref{Hfppftop}, choosing a point in each $\mathscr{L}_x(\rho_1,\rho_2)$ gives a section of $\varpi_x$. Hence $\varpi_x$ is isomorphic to the trivial $\mathbb{Q}/\mathbb{Z}$-torsor over $\mathscr{F}_x$.

Given a finite subset $S\subseteq X^\mathrm{cl}$, we denote
\begin{align*}
    \mathscr{F}_S:=\prod_{x\in S}\mathscr{F}_x=\prod_{x\in S}\left(H^1(G_{K,x},M_1^\vee)\times H^1(G_{K,x},M_2)\right)
\end{align*}
which is finite discrete, being a finite product of finite discrete groups. We will construct a $\mathbb{Q}/\mathbb{Z}$-torsor $\varpi_S:\mathscr{L}_S\rightarrow\mathscr{F}_S$.

If $S=\emptyset$, then this definition yields $\mathscr{F}_\emptyset=0$. In this case, we take the trivial $\mathbb{Q}/\mathbb{Z}$-torsor $\mathscr{L}_\emptyset:=\mathbb{Q}/\mathbb{Z}$.

Assume now that $S\neq\emptyset$. In this case, we will construct $\varpi_S$ from the $\mathbb{Q}/\mathbb{Z}$-torsors $\varpi_x:\mathscr{L}_x\rightarrow\mathscr{F}_x$.

\begin{notation} Let $H$ be an abelian group. Given $H$-sets $\mathscr{L}_1$ and $\mathscr{L}_2$, denote
\begin{align*}
    \mathscr{L}_1\times_H\mathscr{L}_2:=\frac{\mathscr{L}_1\times\mathscr{L}_2}{(\varphi_1+h,\varphi_2)\sim(\varphi_1,\varphi_2+h)}
\end{align*}
This is again an $H$-set under the action
\begin{align*}
    \xymatrix{\mathscr{L}_1\times_H\mathscr{L}_2\times H \ar[r] & \mathscr{L}_1\times_H\mathscr{L}_2 & ([\varphi_1,\varphi_2],h) \ar@{|->}[r] & [\varphi_1+h,\varphi_2]\rlap{\ .}}
\end{align*}

If $\mathscr{L}_1$ and $\mathscr{L}_2$ are $H$-torsors, then $(\varphi_1,\varphi_2)\sim(\widetilde{\varphi}_1,\widetilde{\varphi}_2)$ if and only if
\begin{align*}
    (\varphi_1-\widetilde{\varphi}_1)+(\varphi_2-\widetilde{\varphi}_2)=0
\end{align*}
where $\varphi_i-\widetilde{\varphi}_i\in H$ is the unique element such that $\varphi_i=\widetilde{\varphi}_i+(\varphi_i-\widetilde{\varphi}_i)$. Given $[\varphi_1,\varphi_2],[\psi_1,\psi_2]\in \mathscr{L}_1\times_H\mathscr{L}_2$, we have
\begin{align*}
    [\varphi_1,\varphi_2]&=[\varphi_1+(\psi_1-\varphi_1)+(\psi_2-\varphi_2),\varphi_2]\\
    &=[\varphi_1+(\psi_1-\varphi_1),\varphi_2+(\psi_2-\varphi_2)]=[\psi_1,\psi_2]
\end{align*}
Therefore, $\mathscr{L}_1\times_H\mathscr{L}_2$ is again an $H$-torsor.

Moreover, given $H$-torsors $\mathscr{L}_1,\cdots,\mathscr{L}_n$, we can iterate this construction to get an $H$-torsor $\mathscr{L}_1\times_H\cdots\times_H\mathscr{L}_n$. Then $(\varphi_i)_{i=1,\cdots,n}\sim(\psi_i)_{i=1,\cdots,n}$ if and only if
\begin{align*}
    \sum_{i=1}^n(\varphi_i-\psi_i)=0
\end{align*}
by the above observation.
\end{notation}

Given $\rho_S:=(\rho_{x,1},\rho_{x,2})_{x\in S}\in\mathscr{F}_S$, we apply the above construction to the finite family $(\mathscr{L}_x(\rho_{x,1},\rho_{x,2}))_{x\in S}$ to get a $\mathbb{Q}/\mathbb{Z}$-torsor $\mathscr{L}_S(\rho_S)$. Denote
\begin{align*}
    \mathscr{L}_S:=\coprod_{\rho_S\in\mathscr{F}_S}\mathscr{L}_S(\rho_S)
\end{align*}
the disjoint union. Then the projection to its index set:
\begin{align*}
    \xymatrix{\varpi_S:\mathscr{L}_S \ar[r] & \mathscr{F}_S}
\end{align*}
becomes a $\mathbb{Q}/\mathbb{Z}$-torsor over $\mathscr{F}_S$. Since $\mathscr{F}_S$ is finite discrete, choosing a point in each $\mathscr{L}_S(\rho_S)$ gives a section of $\varpi_S$. Hence $\varpi_S$ is isomorphic to a trivial $\mathbb{Q}/\mathbb{Z}$-torsor over $\mathscr{F}_S$.

\subsection{Local unramified BF functionals}\label{Localthy-BFnr}

For each $x\in X^\mathrm{cl}$, we denote
\begin{eqnarray}\label{fnr}
     \mathscr{F}_x^\mathrm{nr}&:=H^1(G_{K,x}^\mathrm{nr},(M_1^\vee)^{I_{K,x}})\times H^1(G_{K,x}^\mathrm{nr},\pi(M^{I_{K,x}}))\rlap{\ .}
\end{eqnarray}

Since $\pi$ is $G_K$-linear, $\pi(M^{I_{K,x}})\subseteq M_2^{I_{K,x}}$ so the definition makes sense. Since $I_{K,x}$ acts trivially on $\pi(M^{I_{K,x}})$, the inflation map induces an injection
\begin{align*}
    \xymatrix{\mathrm{inf}:H^1(G_{K,x}^\mathrm{nr},\pi(M^{I_{K,x}})) \ar@{^(->}[r] & H^1(G_{K,x},\pi(M^{I_{K,x}}))\rlap{\ .}}
\end{align*}
Since the target is finite by \cite[Theorem 7.1.8]{NSW08}, $\mathscr{F}_x^\mathrm{nr}$ is finite.
From this inclusion and the inflation maps, we get a canonical map of abelian groups:
\begin{align}\label{Fxnr-to-Fx}
    \xymatrixcolsep{3pc}\xymatrix{\mathscr{F}_x^\mathrm{nr} \ar[r] & H^1(G_{K,x}^\mathrm{nr},(M_1^\vee)^{I_{K,x}})\times H^1(G_{K,x}^\mathrm{nr},M_2^{I_{K,x}}) \ar@{^(->}[r]^-{\mathrm{inf}\times\mathrm{inf}} & \mathscr{F}_x\rlap{\ .}}
\end{align}
Note that $\pi(M^{I_{K,x}})=M_2^{I_{K,x}}$ for all but finitely many $x\in X^\mathrm{cl}$. If this is the case, the first map becomes an isomorphism, and hence the composite becomes an injection.

\begin{rmk}
(1)    If one uses the following definition
\begin{align*}
\mathscr{F}_x^\mathrm{nr}&=H^1(G_{K,x}^\mathrm{nr},(M_1^\vee)^{I_{K,x}})\times H^1(G_{K,x}^\mathrm{nr},M_2^{I_{K,x}}),
\end{align*}
then it causes a problem because $M^{I_{K,x}}\rightarrow M_2^{I_{K,x}}$ is not surjective in general, so we do not have the connecting map.

(2) If one uses the following choice
\begin{align*}
\mathscr{F}_x^\mathrm{nr}&=H^1(G_{K,x}^\mathrm{nr},\iota^\vee((M^\vee)^{I_{K,x}}))\times H^1(G_{K,x}^\mathrm{nr},M_2^{I_{K,x}}),
\end{align*}
then it agrees with \eqref{fnr} but beginning with the exact sequence:
\begin{align*}
    \xymatrix{0 \ar[r] & M_2^\vee \ar[r]^-{\pi^\vee} & M^\vee \ar[r]^-{\iota^\vee} & M_1^\vee \ar[r] & 0,}
\end{align*}
where $(-)^\vee=\mathrm{Ab}(-,K_x^{\mathrm{sep},\times})$ is exact.
\end{rmk}

Given a nonempty finite subset $S\subseteq X^\mathrm{cl}$, we denote
\begin{align*}
    \mathscr{F}_S^\mathrm{nr}:=\prod_{x\in S}\mathscr{F}_x^\mathrm{nr}=\prod_{x\in S}\left(H^1(G_{K,x}^\mathrm{nr},(M_1^\vee)^{I_{K,x}})\times H^1(G_{K,x}^\mathrm{nr},\pi(M^{I_{K,x}})\right)\rlap{\ .}
\end{align*}
Taking the product of \eqref{Fxnr-to-Fx} over $S$, we get a canonical map $\mathscr{F}_S^\mathrm{nr}\rightarrow\mathscr{F}_S$. Then we get a pullback torsor:
\begin{align*}
    \xymatrix{
    \mathscr{L}_S^\mathrm{nr} \ar[d]_-{\varpi_S^\mathrm{nr}} \ar[r] & \mathscr{L}_S \ar[d]^-{\varpi_S} \\
    \mathscr{F}_S^\mathrm{nr} \ar[r] & \mathscr{F}_S\rlap{\ .}
    }
\end{align*}
We will construct a map $\mathrm{BF}_S^\mathrm{nr}:\mathscr{F}_S^\mathrm{nr}\rightarrow\mathscr{L}_S$ such that $\varpi_S\circ\mathrm{BF}_S^\mathrm{nr}=(\mathscr{F}_S^\mathrm{nr}\rightarrow\mathscr{F}_S)$, or equivalently a section $\mathrm{BF}_S^\mathrm{nr}:\mathscr{F}_S^\mathrm{nr}\rightarrow\mathscr{L}_S^\mathrm{nr}$ of $\varpi_S^\mathrm{nr}$.

Given $\rho_x=(\rho_{x,1},\rho_{x,2})\in\mathscr{F}_x^\mathrm{nr}$, choose a representative:
\begin{align*}
    (\alpha_{x,1}^\mathrm{nr},\alpha_{x,2}^\mathrm{nr})\in Z^1(G_{K,x}^\mathrm{nr},(M_1^\vee)^{I_{K,x}})\times Z^1(G_{K,x}^\mathrm{nr},\pi(M^{I_{K,x}}))
\end{align*}
We may choose $\sigma\in\mathscr{R}$ such that $\sigma(\pi(M^{I_{K,x}}))\subseteq M^{I_{K,x}}$. For example, choose first a set-theoretic section of $\pi:M^{I_{K,x}}\rightarrow\pi(M^{I_{K,x}})$ and extend it to a set-theoretic section $\sigma$ of $\pi$. Since $M_1$ is finite, we have $M_1^\vee\cong\mathrm{Ab}(M_1,\mathcal{O}_x^{\mathrm{sep},\times})$ so the evaluation pairing fits into the following commutative square:
\begin{align*}
    \xymatrix{
    (M_1^\vee)^{I_{K,x}}\times M_1^{I_{K,x}} \ar@{^(->}[d] \ar[r] & \mathcal{O}_x^{\mathrm{nr},\times} \ar@{^(->}[d] \\
    M_1^\vee\times M_1 \ar[r] & \mathcal{O}_x^{\mathrm{sep},\times}
    }
\end{align*}
and hence we get
\begin{align*}
    \alpha_{x,1}^\mathrm{nr}\cup d(\sigma\circ\alpha_{x,2}^\mathrm{nr})\in Z^3(G_{K,x}^\mathrm{nr},\mathcal{O}_x^{\mathrm{nr},\times})\rlap{\ .}
\end{align*}
Note that $G_{K,x}^\mathrm{nr}\cong\widehat{\mathbb{Z}}$ so by \cite[p.173, Example]{NSW08} and \cite[Proposition 7.1.2]{NSW08}, we have
\begin{align}\label{local-HGnr-vanishing}
    H^q(G_{K,x}^\mathrm{nr},\mathcal{O}_x^{\mathrm{nr},\times})=0\quad\textrm{for $q\geq2$.}
\end{align}

\begin{comment}
{\color{blue}
The above vanishing holds even for $q=1$.
We might add the following as a remark: for any finite unramified extension $L/K_x$ ,
\begin{align*}
 H^q(G(L/K_x),L^\times ) \simeq H^{q-2}(G(L/K_x), \mathbb{Z})\quad\textrm{for $q\geq3$.} \\
\end{align*}
For $n=2$, we have $H^2(G_{K,x}^\mathrm{nr},K_x^{\mathrm{nr},\times}) \simeq \mathbb{Q}/\mathbb{Z}$. 
}

{\color{magenta} I do not know your intention. Freely add the remarks if you want. I stated as \eqref{local-HGnr-vanishing} because we only need the vanishing for $q\geq2$ in defining the unramified BF functional. If you want, replace ``$q\geq2$'' by ``$q\geq1$''.}
\end{comment}

By \eqref{local-HGnr-vanishing} for $q=3$, we have
\begin{align}\label{BFnrx-lift}
    d\varphi_x^\mathrm{nr}=\alpha_{x,1}^\mathrm{nr}\cup d(\sigma\circ\alpha_{x,2}^\mathrm{nr})\quad\textrm{for some}\quad\varphi_x^\mathrm{nr}\in\frac{C^2(G_{K,x}^\mathrm{nr},\mathcal{O}_x^{\mathrm{nr},\times})}{B^2(G_{K,x}^\mathrm{nr},\mathcal{O}_x^{\mathrm{nr},\times})}\rlap{\ .}
\end{align}

By \eqref{local-HGnr-vanishing} for $q=2$, this $\varphi_x^\mathrm{nr}$ is unique. Denote $\varphi_x$ its image under the composition:
\begin{align*}
    \xymatrix{
    \displaystyle\frac{C^2(G_{K,x}^\mathrm{nr},\mathcal{O}_x^{\mathrm{nr},\times})}{B^2(G_{K,x}^\mathrm{nr},\mathcal{O}_x^{\mathrm{nr},\times})} \ar[r] & \displaystyle\frac{C^2(G_{K,x}^\mathrm{nr},K_x^{\mathrm{nr},\times})}{B^2(G_{K,x}^\mathrm{nr},K_x^{\mathrm{nr},\times})} \ar[r] & \displaystyle\frac{C^2(G_{K,x},K_x^{\mathrm{sep},\times})}{B^2(G_{K,x},K_x^{\mathrm{sep},\times})}=\mathrm{ob}(\mathscr{E}_x)\rlap{\ .}
    }
\end{align*}
Given another $\beta_x^\mathrm{nr}=(\beta_{x,1}^\mathrm{nr},\beta_{x,2}^\mathrm{nr})$ and $\tau\in\mathscr{R}$ as above, since $[\alpha_{x,i}^\mathrm{nr}]=\rho_{x,i}=[\beta_{x,i}^\mathrm{nr}]$ for $i=1,2$, we have an arrow in $\mathscr{R}\times\mathscr{A}_x$:
\begin{align*}
    \xymatrix{(\tau-\sigma,g_{x,1},g_{x,2}):(\sigma,\alpha_{x,1}^\mathrm{nr},\alpha_{x,2}^\mathrm{nr}) \ar[r] & (\tau,\beta_{x,1}^\mathrm{nr},\beta_{x,2}^\mathrm{nr})}
\end{align*}
regarding $\alpha_x^\mathrm{nr},\beta_x^\mathrm{nr}$ as objects in $\mathscr{A}_x$ via $(M_1^\vee)^{I_{K,x}}\times\pi(M^{I_{K,x}})\subseteq M_1^\vee\times M_2$. By the functoriality of $\mathbf{d}_\ast\Theta_x$, we have
\begin{align*}
    d(\varphi_x^\mathrm{nr}+\Theta_x(\tau-\sigma,g_1,g_2))&=\alpha_{x,1}^\mathrm{nr}\cup d(\sigma\circ\alpha_{x,2}^\mathrm{nr})+(\mathbf{d}_\ast\Theta_x)(\tau-\sigma,g_1,g_2)\\
    &=\beta_{x,1}^\mathrm{nr}\cup d(\tau\circ\beta_{x,2}^\mathrm{nr})\rlap{\ .}
\end{align*}
This shows that $\varphi_x\bmod\Theta_x\in\mathscr{E}_x/\Theta_x$ is uniquely defined so
\begin{align}\label{FVS-BF-functional}
    \xymatrix{\mathrm{BF}_S^\mathrm{nr}:\mathscr{F}_S^\mathrm{nr} \ar[r] & \mathscr{L}_S & (\rho_x)_{x\in S} \ar@{|->}[r] & [(\varphi_x\bmod\Theta_x)_{x\in S}]\rlap{\ .}}
\end{align}
is a well-defined map. By construction, we have
\begin{align*}
    ([d\varphi_x])_{x\in S}=([\alpha_{x,1}\cup d(\sigma\circ\alpha_{x,2})])_{x\in S}=((\mathbf{d}_\ast\Theta_x)(\rho_x))_{x\in S}
\end{align*}
which shows that $\mathrm{BF}_S^\mathrm{nr}$ defines the desired section.

\section{Classical arithmetic BF theory}\label{Classical-BF}

\subsection{Spaces of fields}\label{ClassicalBF-SoF} Again, let $Y\subseteq X$ be an open subset. As inputs, we need the following data.
\begin{itemize}
    \item An exact sequence of nice commutative group schemes over $Y$ \eqref{BF-input}:
    \begin{align*}
        \xymatrix{0 \ar[r] & \mathcal{M}_1 \ar[r] & \mathcal{M} \ar[r] & \mathcal{M}_2 \ar[r] & 0\rlap{\ .}}
    \end{align*}
    After evaluating at $K^\mathrm{sep}$, we get an exact sequence as in Example \ref{fppf-Gal} which we simply denote
    \begin{align}\label{numberfeq}
    \xymatrix{0 \ar[r] & M_1 \ar[r]^-\iota & M \ar[r]^-\pi & M_2 \ar[r] & 0\rlap{\ .}}
    \end{align}
     From now on, we will freely use the identifications between Galois cohomology and fppf cohomology over fields and discrete valuation rings. For details, see Appendix \ref{fppf-comparison}.
    \item A subgroup of ``Boundary conditions''
    \begin{align*}
        \mathrm{BC}=\prod_{x\in X^\mathrm{cl}}\mathrm{BC}_x\subseteq\prod_{x\in X^\mathrm{cl}}\mathscr{F}_x:=\prod_{x\in X^\mathrm{cl}}H^1(G_{K,x},M_1^\vee)\times H^1(G_{K,x},M_2)
    \end{align*}
    such that $\im(\mathscr{F}_y^\mathrm{nr}\rightarrow\mathscr{F}_y)\subseteq\mathrm{BC}_y\subseteq\mathscr{F}_y$ for $y\in Y^\mathrm{cl}$.
\end{itemize}

Given a finite set $S\subseteq X^\mathrm{cl}$, denote $X_S:= X\setminus S$. This is an affine scheme whose coordinate ring is the ring of $S$-integers:
\begin{align*}
    \Gamma(X_S,\mathcal{O}_{X_S})=\left\{f\in K\ \middle|\ \textrm{$\mathrm{ord}_u(f)\geq0$ for every $x\in X^\mathrm{cl}\setminus S$}\right\}\rlap{\ .}
\end{align*}
We define the restriction of $\mathrm{BC}$ on ``the boundary of $X_S$'' to be
\begin{align*}
    \mathrm{BC}(\partial X_S)=\prod_{x\in X^\mathrm{cl}}\mathrm{BC}(\partial X_S)_x:=\prod_{x\in X_S\setminus Y}\{0\}\times\prod_{x\in S}\mathrm{BC}_x\times\prod_{y\in Y_S^\mathrm{cl}}\mathscr{F}_y^\mathrm{nr}\rlap{\ .}
\end{align*}
Then we define the space of fields $\mathscr{F}(X_S)$ on $X_S$ to be the pullback
\begin{align}\label{FXS-defn}
\begin{aligned}
    \xymatrix{
    \mathscr{F}(X_S) \ar[d] \ar[r] & H^1(G_K,M_1^\vee)\times H^1(G_K,M_2) \ar[d] \\
    \mathrm{BC}(\partial X_S) \ar[r] & \displaystyle\prod_{x\in X^\mathrm{cl}}\mathscr{F}_x\rlap{\ .}
    }
\end{aligned}
\end{align}

As a sanity check, we consider the case $S=\emptyset$ which gives
\begin{align*}
    \mathrm{BC}(\partial X)=\prod_{x\in X\setminus Y}\{0\}\times\prod_{y\in Y^\mathrm{cl}}\mathscr{F}_y^\mathrm{nr}\rlap{\ .}
\end{align*}
\begin{lem}\label{FX} We have a canonical map
\begin{align}\label{FX-to-D1D1}
    \xymatrix{\mathscr{F}(X) \ar[r] & D_\mathrm{fppf}^1(Y,\mathcal{M}_1^\vee)\times D_\mathrm{fppf}^1(Y,\mathcal{M}_2)\rlap{\ .}}
\end{align}
This map becomes an isomorphism if $\pi(M^{I_{K,y}})=M_2^{I_{K,y}}$ for every $y\in Y^\mathrm{cl}$. In particular, if the order of $\mathcal{M}$ is invertible on $Y$, then the canonical map becomes an isomorphism.
\end{lem}
\begin{proof} Let $\mathcal{G}$ be a nice commutative group scheme over $Y$. By \cite[Proposition 2.2]{Kes17}, we have pullback squares:
\begin{align*}
\begin{aligned}
    \xymatrix{
    H_\mathrm{fppf}^1(Y,\mathcal{G}) \ar[d] \ar@{^(->}[r] & H_\mathrm{fppf}^1(K,\mathcal{G}) \ar[d] \\
    \displaystyle\prod_{x\in X\setminus Y}H_\mathrm{fppf}^1(K_x,\mathcal{G})\times\prod_{y\in Y^\mathrm{cl}}H_\mathrm{fppf}^1(\mathcal{O}_u,\mathcal{G}) \ar[d] \ar@{^(->}[r] & \displaystyle\prod_{x\in X^\mathrm{cl}}H_\mathrm{fppf}^1(K_x,\mathcal{G}) \ar[d]^-{\mathrm{pr}_{Y^\mathrm{cl}}} \\
    \displaystyle\prod_{y\in Y^\mathrm{cl}}H_\mathrm{fppf}^1(\mathcal{O}_y,\mathcal{G}) \ar@{^(->}[r] & \displaystyle\prod_{y\in Y^\mathrm{cl}}H_\mathrm{fppf}^1(K_y,\mathcal{G})
}
\end{aligned}
\end{align*}
where $\mathrm{pr}_{Y^\mathrm{cl}}$ discards the $X\setminus U$ components. Recall that we have an exact sequence:
\begin{align*}
    \xymatrix{0 \ar[r] & D_\mathrm{fppf}^1(Y,\mathcal{G}) \ar[r] & H_\mathrm{fppf}^1(Y,\mathcal{G}) \ar[r] & \displaystyle\bigoplus_{x\in X\setminus Y}H_\mathrm{fppf}^1(K_x,\mathcal{G})\rlap{\ .}}
\end{align*}
Since $X\setminus Y$ is finite, if we choose
\begin{align*}
    \mathcal{W}(Y,\mathcal{G}):=\prod_{x\in X\setminus Y}\{0\}\times\prod_{y\in Y^\mathrm{cl}}H_\mathrm{fppf}^1(\mathcal{O}_y,\mathcal{G})\subseteq\prod_{x\in X^\mathrm{cl}}H_\mathrm{fppf}^1(K_x,\mathcal{G})
\end{align*}
then the above exact sequence fits into the following pullback squares:
\begin{align*}
    \xymatrixcolsep{1pc}\xymatrix{
    D_\mathrm{fppf}^1(Y,\mathcal{G}) \ar[d] \ar@{^(->}[r] & H_\mathrm{fppf}^1(Y,\mathcal{G}) \ar[d] \ar@{^(->}[r] & H_\mathrm{fppf}^1(K,\mathcal{G}) \ar[d] \\
    \mathcal{W}(Y,\mathcal{G}) \ar[d] \ar@{^(->}[r] & \displaystyle\prod_{x\in X\setminus Y}H_\mathrm{fppf}^1(K_x,\mathcal{G})\times\prod_{y\in Y^\mathrm{cl}}H_\mathrm{fppf}^1(\mathcal{O}_y,\mathcal{G}) \ar[d] \ar@{^(->}[r] & \displaystyle\prod_{x\in X^\mathrm{cl}}H_\mathrm{fppf}^1(K_x,\mathcal{G}) \\
    0 \ar@{^(->}[r] & \displaystyle\prod_{y\in X\setminus Y}H_\mathrm{fppf}^1(K_x,\mathcal{G})\rlap{\ .} &
}
\end{align*}
Now return to the input data. With the above notation, we have a canonical map
\begin{align}\label{BC-to-WW}
    \xymatrix{\mathrm{BC}(\partial X) \ar[r] & \mathcal{W}(Y,\mathcal{M}_1^\vee)\times\mathcal{W}(Y,\mathcal{M}_2)}
\end{align}
which induces the asserted map by taking the pullbacks: 
\begin{align*}
    \xymatrixcolsep{1.5pc}\xymatrix{
    \mathscr{F}(X) \ar[d] \ar[r] & D_\mathrm{fppf}^1(Y,\mathcal{M}_1^\vee)\times D_\mathrm{fppf}^1(Y,\mathcal{M}_2) \ar[d] \ar@{^(->}[r] & H_\mathrm{fppf}^1(K,\mathcal{M}_1^\vee)\times H_\mathrm{fppf}^1(K,\mathcal{M}_2) \ar[d] \\
    \mathrm{BC}(\partial X) \ar[r]^-{\eqref{BC-to-WW}} & \mathcal{W}(Y,\mathcal{M}_1^\vee)\times\mathcal{W}(Y,\mathcal{M}_2) \ar@{^(->}[r] & \displaystyle\prod_{x\in X^\mathrm{cl}}\mathscr{F}_x\rlap{\ .}
    }
\end{align*}
If $\pi(M^{I_{K,y}})=M_2^{I_{K,y}}$ for every $y\in Y^\mathrm{cl}$, then \eqref{BC-to-WW} becomes an isomorphism and so does the induced map.
\end{proof}
\begin{rmk}\label{D-res-inj} Let $\mathcal{G}$ be a nice commutative group scheme over $Y$. By the proof of Lemma \ref{FX}, for open subsets $U\subseteq V\subseteq Y$ the restriction map of $H_\mathrm{fppf}^1(-,\mathcal{G})$ fits into the following pullback square:
\begin{align*}
    \xymatrix{
    H_\mathrm{fppf}^1(V,\mathcal{G}) \ar[d] \ar@{^(->}[r] & H_\mathrm{fppf}^1(U,\mathcal{G}) \ar[d] \\
    \displaystyle\prod_{x\in X\setminus V}H_\mathrm{fppf}^1(K_x,\mathcal{G})\times\prod_{v\in V\setminus U}H_\mathrm{fppf}^1(\mathcal{O}_v,\mathcal{G}) \ar@{^(->}[r] & \displaystyle\prod_{x\in X\setminus U}H_\mathrm{fppf}^1(K_x,\mathcal{G})
    }
\end{align*}
On the other hand, by the naturality of the inclusion $D_\mathrm{fppf}^1(-,\mathcal{G})\subseteq H_\mathrm{fppf}^1(-,\mathcal{G})$, we have the following commutative square:
\begin{align*}
    \xymatrix{
    D_\mathrm{fppf}^1(V,\mathcal{G}) \ar@{^(->}[d] \ar[r] & D_\mathrm{fppf}^1(U,\mathcal{G}) \ar@{^(->}[d] \\
    H_\mathrm{fppf}^1(V,\mathcal{G}) \ar@{^(->}[r] & H_\mathrm{fppf}^1(U,\mathcal{G})\rlap{\ .}
    }
\end{align*}
This shows that the restriction maps of $D_\mathrm{fppf}^1(-,\mathcal{G})$ along the inclusions $U\subseteq V$ of open subsets of $Y$ are injective.
\end{rmk}

\subsection{Global arithmetic BF functionals}\label{Global-BF}

For each finite subset $S\subseteq X^\mathrm{cl}$, we have a pullback torsor:
\begin{align*}
    \xymatrix{
    \partial_S^\ast\mathscr{L}_S \ar[d] \ar[r] & \mathscr{L}_S \ar[d]^-{\varpi_S} \\
    \mathscr{F}(X_S) \ar[r]_-{\partial_S} & \mathscr{F}_S\rlap{\ .}
    }
\end{align*}
We will construct a map $\mathrm{BF}_{X_S}:\mathscr{F}(X_S)\rightarrow\mathscr{L}_S$ such that $\varpi_S\circ\mathrm{BF}_{X_S}=\partial_S$, or equivalently a section $\mathrm{BF}_{X_S}:\mathscr{F}(X_S)\rightarrow\partial_S^\ast\mathscr{L}_S$ of the pullback torsor.

If $S=\emptyset$, then from Lemma \ref{FX} we have the following composite:
\begin{align*}
    \xymatrix{\mathrm{BF}_X:\mathscr{F}(X) \ar[r]^-{\eqref{FX-to-D1D1}} & D_\mathrm{fppf}^1(Y,\mathcal{M}_1^\vee)\times D_\mathrm{fppf}^1(Y,\mathcal{M}_2) \ar[r]^-{\eqref{D1D1-pairing}} & \mathbb{Q}/\mathbb{Z}\rlap{\ .}}
\end{align*}
Since $\partial_\emptyset^\ast\mathscr{L}_\emptyset\cong\mathscr{F}(X)\times\mathbb{Q}/\mathbb{Z}$, the above $\mathrm{BF}_X$ correspnds to
\begin{align*}
    \xymatrix{(\mathrm{Id}_{\mathscr{F}(X)},\mathrm{BF}_X):\mathscr{F}(X) \ar[r] & \mathscr{F}(X)\times\mathbb{Q}/\mathbb{Z}}
\end{align*}

If $S\neq\emptyset$, then the procedure is analogous to what we did in Section \ref{Localthy-BFnr}. We first introduce notations.

\begin{notation}\label{bdyS} For each $x\in X^\mathrm{cl}$, we have an embedding $G_{K,x}\hookrightarrow G_K$ induced from the chosen embedding $K^\mathrm{sep}\hookrightarrow K_x^\mathrm{sep}$. Then denote
\begin{align*}
    \xymatrix{\partial_x:=(G_{K,x}\hookrightarrow G_K)^\ast:C^\bullet(G_K,-) \ar[r] & C^\bullet(G_{K,x},-)\rlap{\ .}}
\end{align*}
Analogously, given a nonempty finite subset $S\subseteq X^\mathrm{cl}$, we denote
\begin{align*}
    \xymatrix{\partial_S:=(\partial_x)_{x\in S}:C^\bullet(G_K,-) \ar[r] & \displaystyle\prod_{x\in S}C^\bullet(G_{K,x},-)\rlap{\ .}}
\end{align*}
We use the same notations for the induced maps on cohomology groups. For example, applying this to \eqref{BF-input-GK} we get
\begin{align*}
    \xymatrix{\partial_S:H^1(G_K,M_1^\vee)\times H^1(G_K,M_2) \ar[r] & \mathscr{F}_S\rlap{\ .}}
\end{align*}
We use the same notation for its restriction to subgroups or composition with \eqref{FXS-defn}. Then each $\partial_x:\mathscr{F}(X_S)\rightarrow\mathscr{F}_x$ admits a factorization:
\begin{align*}
    \xymatrix{
    \mathscr{F}(X_S) \ar[d]_-{\eqref{FXS-defn}} \ar[r]^-{\partial_x} & \mathscr{F}_x \\
    \mathrm{BC}(\partial X_S) \ar[r]_-{\mathrm{pr}_x} & \mathrm{BC}(\partial X_S)_x\rlap{\ .} \ar[u]
    }
\end{align*}
If $x\in X_S^\mathrm{cl}$, then $\mathrm{BC}(\partial X_S)_x\subseteq\mathscr{F}_x$ by construction. In this case, we denote
\begin{align*}
    \xymatrix{\partial_x^\mathrm{nr}:\mathscr{F}(X_S) \ar[r]^-{\eqref{FXS-defn}} & \mathrm{BC}(\partial X_S) \ar[r]^-{\mathrm{pr}_x} & \mathrm{BC}(\partial X_S)_x\rlap{\ .}}
\end{align*}
Analogously, given a nonempty finite subset $\Lambda\subseteq X_S^\mathrm{cl}$, we denote
\begin{align*}
    \xymatrix{\partial_\Lambda^\mathrm{nr}:=(\partial_x^\mathrm{nr})_{x\in\Lambda}:\mathscr{F}(X_S) \ar[r] & \mathrm{BC}(\partial X_S)_\Lambda\rlap{\.}}
\end{align*}
\end{notation}

\begin{rmk}\label{FBC-pullback} Let $S\subseteq T\subseteq X^\mathrm{cl}$ be finite subsets. In this case, we may rewrite
\begin{align*}
    \mathrm{BC}(\partial X_S)&=\prod_{x\in X_S\setminus Y}\{0\}\times\prod_{x\in S}\mathrm{BC}_x\times\prod_{y\in Y_S^\mathrm{cl}}\mathscr{F}_y^\mathrm{nr}\\
    &=\prod_{x\in X_T\setminus Y}\{0\}\times\prod_{x\in(T\setminus S)\setminus Y}\{0\}\times\prod_{x\in S}\mathrm{BC}_x\times\prod_{y\in(T\setminus S)\cap Y}\mathscr{F}_y^\mathrm{nr}\times\prod_{y\in Y_T^\mathrm{cl}}\mathscr{F}_y^\mathrm{nr}
\end{align*}
so we have a canonical map $\mathrm{BC}(\partial X_S)\rightarrow\mathrm{BC}(\partial X_T)$ by our assumption on $\mathrm{BC}$. Then we have an induced map $\mathscr{F}(X_S)\rightarrow\mathscr{F}(X_T)$ which will be our restriction map. More precisely, we have the following commutative cube:
\begin{align*}
    \xymatrix@!0@R=3.5pc@C=5pc{
    & \mathrm{BC}(\partial X_S) \ar[dd]|\hole^(.25){\mathrm{pr}_T} \ar[rr] & & \mathrm{BC}(\partial X_T) \ar[dd]^(.25){\mathrm{pr}_T} \\
    \mathscr{F}(X_S) \ar[dd]_(.25){(\partial_S,\partial_{T\setminus S}^\mathrm{nr})} \ar[rr] \ar[ur] & & \mathscr{F}(X_T) \ar[dd]^(.25){\partial_T} \ar[ur] & \\
    & \mathrm{BC}_S\times\mathrm{BC}(\partial X_S)_{T\setminus S} \ar@{^(->}[dl] \ar[rr]|\hole & & \mathrm{BC}_T \ar@{^(->}[dl] \\
    \mathscr{F}_S\times\mathrm{BC}(\partial X_S)_{T\setminus S} \ar[rr] & & \mathscr{F}_T &
    }
\end{align*}
whose top, back, and bottom faces are pullback squares by definition. Hence the front face becomes a pullback square.
\end{rmk}

Given $\rho=(\rho_1,\rho_2)\in\mathscr{F}(X_S)$, choose a representative of its image under \eqref{FXS-defn}:
\begin{align*}
    (\alpha_1,\alpha_2)\in Z^1(G_K,M_1^\vee)\times Z^1(G_K,M_2)
\end{align*}
and a set-theoretic section $\sigma:M_2\rightarrow M$ of $\pi$, i.e., $\sigma\in\mathscr{R}$. By \cite[Corollary I.4.18]{Mil06}, for any positive prime $\ell$, the $\ell$-primary subgroup of $H^3(G_K,K^{\mathrm{sep},\times})$ vanishes:
\begin{align*}
    H^3(G_K,K^{\mathrm{sep},\times})(\ell)=0\rlap{\ .}
\end{align*}
Since $M$ is finite, $\alpha_1\cup d(\sigma\circ\alpha_2)$ is torsion so we have
\begin{align*}
    d\varphi=\alpha_1\cup d(\sigma\circ\alpha_2)\quad\textrm{for some}\quad\varphi\in\frac{C^2(G_K,K^{\mathrm{sep},\times})}{B^2(G_K,K^{\mathrm{sep},\times})}\rlap{\ .}
\end{align*}
Given another choice $(\beta_1,\beta_2)$ and $\tau\in\mathscr{R}$, we have
\begin{align*}
    \beta_i=\alpha_i+dg_i\quad\textrm{for some}\quad(g_1,g_2)\in C^0(G_K,M_1^\vee)\times C^0(G_K,M_2)
\end{align*}
Then, for each $x\in X^\mathrm{cl}$, we have
\begin{align*}
    d(\partial_x\varphi+\Theta_x(\tau-\sigma,\partial_xg_1,\partial_xg_2))=\partial_x(\beta_1\cup d(\tau\circ\beta_2))
\end{align*}
analogously as in the construction of \eqref{FVS-BF-functional}. Note that the $\Theta_x$ defined in Lemma \ref{cup-functor}, and $\partial_x$ is defined in Notation \ref{bdyS}. This shows that
\begin{align}\label{FUS-BF-functional}
    \xymatrix{\mathrm{BF}_{X_S}:\mathscr{F}(X_S) \ar[r] & \mathscr{L}_S & \rho \ar@{|->}[r] & [(\partial_x\varphi\bmod\Theta_x)_{x\in S}]}
\end{align}
is a well-defined map. By construction, we have
\begin{align*}
    ([d(\partial_x\varphi)])_{x\in S}=([\partial_x\alpha_1\cup d(\sigma\circ\partial_x\alpha_2)])_{x\in S}=((\mathbf{d}_\ast\Theta_x)(\partial_x\rho))_{x\in S}
\end{align*}
which shows that $\mathrm{BF}_{X_S}$ defines the desired section.

Given $\xi\in\Gamma(\mathscr{F}_S,\mathscr{L}_S)$, we have a trivialization of the $\mathbb{Q}/\mathbb{Z}$-torsor:
\begin{align}\label{scrLS-trivialization}
\begin{aligned}
    &\xymatrix{\Phi_S^\xi:\mathscr{L}_S \ar[r]^-\sim & \mathscr{F}_S\times\mathbb{Q}/\mathbb{Z}} \\
    &\Phi_S^\xi[(\lambda_x)_{x\in S}]:=\left((\varpi_x(\lambda_x))_{x\in S},\sum_{x\in S}\mathrm{inv}_x(\lambda_x-(\xi_x\circ\varpi_x)(\lambda_x))\right)
\end{aligned}
\end{align}
which induces the following trivializations of the pullback torsors:
\begin{align*}
    \xymatrix{\partial_S^\ast\Phi_S^\xi:\partial_S^\ast\mathscr{L}_S \ar[r]^-\sim & \mathscr{F}_S\times\mathbb{Q}/\mathbb{Z} & \Phi_S^{\xi,\mathrm{nr}}:\mathscr{L}_S^\mathrm{nr} \ar[r]^-\sim & \mathscr{F}_S^\mathrm{nr}\times\mathbb{Q}/\mathbb{Z}\rlap{\ .}}
\end{align*}
Each of the above isomorphisms induces a group isomorphism of the form:
\begin{align*}
    \xymatrix{\left(\mathrm{pr}_{\mathbb{Q}/\mathbb{Z}}\circ\Phi^\xi\right)_\ast:\Gamma(\mathscr{F},\mathscr{L}) \ar[r]^-\sim & \mathrm{Set}(\mathscr{F},\mathbb{Q}/\mathbb{Z})}
\end{align*}
for each of the following triples:
\begin{align*}
    \left(\Phi^\xi,\mathscr{F},\mathscr{L}\right)\in\left\{\left(\Phi_S^\xi,\mathscr{F}_S,\mathscr{L}_S\right),\left(\partial_S^\ast\Phi_S^\xi,\mathscr{F}(X_S),\partial_S^\ast\mathscr{L}_S\right),\left(\Phi_S^{\xi,\mathrm{nr}},\mathscr{F}_S^\mathrm{nr},\mathscr{L}_S^\mathrm{nr}\right)\right\}\rlap{\ .}
\end{align*}
Along the above isomorphism, we get functions
\begin{align}\label{secBF}
    \mathrm{BF}_{X_S}^\xi:=\mathrm{pr}_{\mathbb{Q}/\mathbb{Z}}\circ\partial_S^\ast\Phi_S^\xi\circ\mathrm{BF}_{X_S}\in\mathrm{Set}(\mathscr{F}(X_S),\mathbb{Q}/\mathbb{Z})
\end{align}
\begin{align*}
    \mathrm{BF}_S^{\mathrm{nr},\xi}:=\mathrm{pr}_{\mathbb{Q}/\mathbb{Z}}\circ\Phi_S^\xi|_{\mathscr{F}_S^\mathrm{nr}}\circ\mathrm{BF}_S^\mathrm{nr}\in\mathrm{Set}(\mathscr{F}_S^\mathrm{nr},\mathbb{Q}/\mathbb{Z})
\end{align*}
corresponding to the sections defined above.

\begin{notation}\label{Bdy-reversed} Denote $\varpi_{S^\ast}:\mathscr{L}_{S^\ast}\rightarrow\mathscr{F}_S$ the $\mathbb{Q}/\mathbb{Z}$-torsor over $\mathscr{F}_S$ constructed as in Section \ref{Classical-BF} but using the following isomorphism:
\begin{align*}
    \xymatrix{-\mathrm{inv}_x:H^2(G_{K,X},K_x^{\mathrm{sep},\times}) \ar[r]^-\sim & \mathbb{Q}/\mathbb{Z}}
\end{align*}
so that $\mathscr{L}_{S^\ast}$ has the same underlying set as $\mathscr{L}_S$ but the $\mathbb{Q}/\mathbb{Z}$-action is twisted by $-1:\mathbb{Q}/\mathbb{Z}\rightarrow\mathbb{Q}/\mathbb{Z}$. Also, for each $\xi\in\Gamma(\mathscr{F}_S,\mathscr{L}_S)$, the trivialization map becomes
\begin{align*}
    \xymatrix{\Phi_{S^\ast}^\xi=-\Phi_S^\xi:\mathscr{L}_{S^\ast} \ar[r]^-\sim & \mathscr{F}_S\times\mathbb{Q}/\mathbb{Z}\rlap{\ .}}
\end{align*}
This gives the following relations
\begin{align*}
    \mathrm{BF}_{X_{S^\ast}}^\xi=-\mathrm{BF}_{X_S}^\xi,\quad\mathrm{BF}_{S^\ast}^{\mathrm{nr},\xi}=-\mathrm{BF}_S^{\mathrm{nr},\xi}\rlap{\ .}
\end{align*}
\end{notation}

\subsection{Decomposition formula} Given finite subsets $S\subseteq T\subseteq X^\mathrm{cl}$, denote
\begin{align*}
    \xymatrix{\oplus:\mathscr{F}_S\times\mathscr{F}_{T\setminus S} \ar[r]^-\sim & \mathscr{F}_T}
\end{align*}
the canonical isomorphism which is characterized by the property
\begin{align*}
    \beta=\mathrm{pr}_S\beta\oplus\mathrm{pr}_{T\setminus S}\beta\quad\textrm{for}\quad\beta\in\mathscr{F}_T\rlap{\ .}
\end{align*}

Assume first that $S\neq\emptyset$. For each $\beta\in\mathscr{F}_T$, the canonical bijection as above:
\begin{align*}
    \xymatrix{\displaystyle\oplus:\prod_{x\in S}\mathscr{L}_x(\beta_x)\times\prod_{x\in T\setminus S}\mathscr{L}_x(\beta_x) \ar[r]^-\sim & \displaystyle\prod_{x\in T}\mathscr{L}_x(\beta_x)}
\end{align*}
induces a $\mathbb{Q}/\mathbb{Z}$-bi-equivariant quotient map:
\begin{align*}
    \xymatrix{\oplus:\mathscr{L}_S(\mathrm{pr}_S\beta)\times\mathscr{L}_{T\setminus S}(\mathrm{pr}_{T\setminus S}\beta) \ar[r] & \mathscr{L}_T(\beta)\rlap{\ .}}
\end{align*}
Gathering these maps, we get
\begin{align}\label{scrL-sum}
    \xymatrix{\oplus:\mathscr{L}_S\times\mathscr{L}_{T\setminus S} \ar[r] & \mathscr{L}_T}
\end{align}
which fits into the following commutative square:
\begin{align*}
    \xymatrix{
    \mathscr{L}_S\times\mathscr{L}_{T\setminus S} \ar[d]_-{\varpi_S\times\varpi_{T\setminus S}} \ar[r]^-\oplus & \mathscr{L}_T \ar[d]^-{\varpi_T} \\
    \mathscr{F}_S\times\mathscr{F}_{T\setminus S} \ar[r]^-\sim_-\oplus & \mathscr{F}_T\rlap{\ .}
    }
\end{align*}
Note that the composite $\mathscr{L}_S\times\mathscr{L}_{T\setminus S}\rightarrow\mathscr{F}_T$ is identified with the external product of $\varpi_S$ and $\varpi_{T\setminus S}$.

If $S=\emptyset$, then $\oplus$ is defined to be the $\mathbb{Q}/\mathbb{Z}$-action on $\mathscr{L}_T$:
\begin{align*}
    \xymatrix{\oplus:\mathscr{L}_\emptyset\times\mathscr{L}_T \ar@{=}[r] & \mathbb{Q}/\mathbb{Z}\times\mathscr{L}_T \ar[r]^-{\mathrm{action}} & \mathscr{L}_T\rlap{\ .}}
\end{align*}
Since $\varpi_T$ is $\mathbb{Q}/\mathbb{Z}$-invariant, this fits into the following commutative square:
\begin{align*}
    \xymatrix{
    \mathscr{L}_\emptyset\times\mathscr{F}_T \ar[d]_-{\varpi_\emptyset\times\mathscr{L}_T} \ar[r]^-\oplus & \mathscr{L}_T \ar[d]^-{\varpi_T} \\
    \mathscr{F}_\emptyset\times\mathscr{F}_T \ar[r]^-\sim & \mathscr{F}_T\rlap{\ .}
    }
\end{align*}

\begin{thm}[Decomposition formula]\label{Decomp-formula} Let $S\subseteq T\subseteq X^\mathrm{cl}$ be finite subsets.
\begin{quote}
    (1) If $S\neq\emptyset$, then the following equality holds in $\Gamma(\mathscr{F}(X_S),\partial_T^\ast\mathscr{L}_T)$:
    \begin{align*}
        \mathrm{BF}_{X_T}|_{\mathscr{F}(X_S)}=\mathrm{BF}_{X_S}\oplus\left(\mathrm{BF}_{T\setminus S}^\mathrm{nr}\circ\partial^\mathrm{nr}_{T\setminus S}\right)\rlap{\ .}
    \end{align*}
    (2) If $S=\emptyset$, $X\setminus Y\subseteq T$, and the order of $\mathcal{M}$ is invertible on $Y_T$, then the following equality holds in $\Gamma(\mathscr{F}(X),\partial_T^\ast\mathscr{L}_T)$:
    \begin{align*}
        \mathrm{BF}_{X_T}|_{\mathscr{F}(X)}=\mathrm{BF}_X\oplus\left(\mathrm{BF}_T^\mathrm{nr}\circ\partial_T^\mathrm{nr}\right)\rlap{\ .}
    \end{align*}
\end{quote}
\end{thm}
\begin{proof} (1) Given $\rho=(\rho_1,\rho_2)\in\mathscr{F}(X_S)$, choose a representative of its image under \eqref{FXS-defn}:
\begin{align*}
    \alpha=(\alpha_1,\alpha_2)\in Z^1(G_K,M_1^\vee)\times Z^1(G_K,M_2)
\end{align*}
For each $x\in X_S\setminus X_T=T\setminus S$, we have $\partial_x\rho\in\mathscr{F}_x^\mathrm{nr}$ or $\partial_x\rho=0$ by Remark \ref{FBC-pullback} so we may choose its unramified representative:
\begin{align*}
    \alpha_x^\mathrm{nr}=(\alpha_{x,1}^\mathrm{nr},\alpha_{x,2}^\mathrm{nr})\in Z^1(G_{K,x}^\mathrm{nr},(M_1^\vee)^{I_{K,x}})\times Z^1(G_{K,x}^\mathrm{nr},\pi(M^{I_{K,x}}))
\end{align*}
and hence $\partial_x\alpha\cong\alpha_x^\mathrm{nr}$ in $\mathscr{A}_x$. Choose $\sigma\in\mathscr{R}$ such that $\sigma(\pi(M^{I_{K,x}}))\subseteq M^{I_{K,x}}$. Then $(\sigma,\partial_x\alpha)\cong(\sigma,\alpha_x)$ in $\mathscr{R}\times\mathscr{A}_x$. Choose $\varphi$ as in \eqref{FUS-BF-functional} and $\varphi_x$ as in \eqref{FVS-BF-functional}. Then the following holds in $\mathscr{E}_x/\Theta_x$:
\begin{align*}
    \partial_x\varphi\bmod\Theta_x=\varphi_x\bmod\Theta_x\rlap{\ .}
\end{align*}
From the above observation, we get
\begin{align*}
    \mathrm{BF}_{X_T}(\rho)&=[(\partial_x\varphi\bmod\Theta_x)_{x\in S}]\oplus[(\partial_x\varphi\bmod\Theta_x)_{x\in T\setminus S}]\\
    &=\mathrm{BF}_{X_S}(\rho)\oplus[(\varphi_x\bmod\Theta_x)_{x\in T\setminus S}]\\
    &=\mathrm{BF}_{X_S}(\rho)\oplus\mathrm{BF}_{T\setminus S}^\mathrm{nr}(\partial_{T\setminus S}^\mathrm{nr}\rho)\rlap{\ .}
\end{align*}
\indent (2) In this case, \eqref{BF-input} becomes an exact sequence finite \'etale group schemes over $Y_T$. By Remark \ref{et-fppf-smqp} and \cite[Proposition II.4.13]{Mil06}, $\mathrm{BF}_X$ can be identified with the corresponding pairing of Galois cohomology groups over $Y_T$:
\begin{align*}
    \xymatrix{
    \mathrm{BF}_X:\mathscr{F}(X) \ar[r]^-{\eqref{FX-to-D1D1}} & D_\mathrm{fppf}^1(Y,\mathcal{M}_1^\vee)\times D_\mathrm{fppf}^1(Y,\mathcal{M}_2) \ar@{^(->}[d] \ar[r]^-{\eqref{D1D1-pairing}} & \mathbb{Q}/\mathbb{Z} \ar@{=}[d] \\
    & D_\mathrm{fppf}^1(Y_T,\mathcal{M}_1^\vee)\times D_\mathrm{fppf}^1(Y_T,\mathcal{M}_2) \ar[d]^-\wr \ar[r] & \mathbb{Q}/\mathbb{Z} \ar@{=}[d] \\
    & \Sha_T^1(K,M_1^\vee)\times\Sha_T^1(K,M_2) \ar[r] & \mathbb{Q}/\mathbb{Z}\rlap{\ .}
    }
\end{align*}
Given $\rho\in\mathscr{F}(X)$, we compute $\mathrm{BF}_X(\rho)$ along the bottom row. Replacing $Y$ by $Y_T$ if necessary, we may assume that $T=X\setminus Y$. In this case, we may regard $\partial_x\rho=0$ for every $x\in T$. Then we may choose its unramified representative in the coboundary subgroup:
\begin{align*}
    \alpha_x^\mathrm{nr}=(\alpha_{x,1}^\mathrm{nr},\alpha_{x,2}^\mathrm{nr})\in B^1(G_{K,x}^\mathrm{nr},(M_1^\vee)^{I_{K,x}})\times B^1(G_{K,x}^\mathrm{nr},\pi(M^{I_{K,x}}))\rlap{\ .}
\end{align*}
Choose $\mu_{2,x}\in C^0(G_{K,x}^\mathrm{nr},\pi(M^{I_{K,x}}))$ such that $\alpha_{x,2}^\mathrm{nr}=d\mu_{2,x}$. Choose $\sigma\in\mathscr{R}$ such that $\sigma(\pi(M^{I_{K,x}}))\subseteq M^{I_{K,x}}$. Since $\pi$ is $G_K$-linear, it commutes with the differentials on $C^\bullet(G_{K,x}^\mathrm{nr},-)$ so we have
\begin{align*}
    \pi d(\sigma\circ\mu_{2,x})=d(\pi\circ\sigma\circ\mu_{2,x})=d\mu_{2,x}=\alpha_{2,x}^\mathrm{nr}\rlap{\ .}
\end{align*}
This implies that
\begin{align*}
    d(\sigma\circ\mu_{2,x})-\sigma\circ\alpha_{2,x}^\mathrm{nr}\in C^1(G_{K,x}^\mathrm{nr},M_1^{I_{K,x}})\rlap{\ .}
\end{align*}
Now the cochain
\begin{align*}
    \alpha_{x,1}^\mathrm{nr}\cup(d(\sigma\circ\mu_{2,x})-\sigma\circ\alpha_{2,x}^\mathrm{nr})\in C^2(G_{K,x},\mathcal{O}_{K,x}^{\mathrm{nr},\times})
\end{align*}
satisfies
\begin{align*}
    d\left(\alpha_{x,1}^\mathrm{nr}\cup(d(\sigma\circ\mu_{2,x})-\sigma\circ\alpha_{2,x}^\mathrm{nr})\right)=\alpha_{x,1}^\mathrm{nr}\cup d(\sigma\circ\alpha_{2,x}^\mathrm{nr})\rlap{\ .}
\end{align*}
By the uniqueness coming from \eqref{local-HGnr-vanishing} for $q=2$, we have
\begin{align*}
    \varphi_x^\mathrm{nr}=\alpha_{x,1}^\mathrm{nr}\cup(d(\sigma\circ\mu_{2,x})-\sigma\circ\alpha_{2,x}^\mathrm{nr})\bmod B^2(G_{K,x},\mathcal{O}_{K,x}^{\mathrm{nr},\times})\rlap{\ .}
\end{align*}
Then we have (cf. \cite[p.65]{Mil06}\footnote{There are mistakes on the sign. Here we correct the sign.}) the following equality in $\mathbb{Q}/\mathbb{Z}$:
\begin{align*}
    \mathrm{BF}_X(\rho)=\sum_{x\in T}\mathrm{inv}_x(\partial_x\varphi-\varphi_x)=\mathrm{BF}_{X_T}(\rho)-\mathrm{BF}_T^\mathrm{nr}(\partial^\mathrm{nr}_T\rho)\rlap{\ .}
\end{align*}
By definition of $\oplus$ for $S=\emptyset$ case, this is exactly what we have to prove.
\end{proof}

Once we choose sections properly, $\oplus$ in the above decomposition formula becomes the sum of sections. For this we introduce the following notation: Given finite subsets $S\subseteq T\subseteq X^\mathrm{cl}$, denote
\begin{align*}
    \xymatrixcolsep{0.9pc}\xymatrix{\boxplus:\Gamma(\mathscr{F}_S,\mathscr{L}_S)\times\Gamma(\mathscr{F}_{T\setminus S},\mathscr{L}_{T\setminus S}) \ar[r] & \Gamma(\mathscr{F}_T,\mathscr{L}_T) & (\alpha,\gamma) \ar@{|->}[r] & (\alpha\circ\mathrm{pr}_S)\oplus(\gamma\circ\mathrm{pr}_{T\setminus S})}
\end{align*}
where $\oplus$ is as in \eqref{scrL-sum}.

\begin{prop}\label{Decomp-formula-with-section} Let $S\subseteq T\subseteq X^\mathrm{cl}$ be finite subsets. Given
\begin{align*}
    (\xi_S,\xi_{T\setminus S})\in\Gamma(\mathscr{F}_S,\mathscr{L}_S)\times\Gamma(\mathscr{F}_{T\setminus S},\mathscr{L}_{T\setminus S})
\end{align*}
the following holds.
\begin{quote}
    (1) If $S\neq\emptyset$, then the following equality holds in $\mathrm{Set}(\mathscr{F}(X_S),\mathbb{Q}/\mathbb{Z})$:
    \begin{align*}
        \mathrm{BF}_{X_T}^{\xi_S\boxplus\xi_{T\setminus S}}|_{\mathscr{F}(X_S)}=\mathrm{BF}_{X_S}^{\xi_S}+\left(\mathrm{BF}_{T\setminus S}^{\mathrm{nr},\xi_{T\setminus S}}\circ\partial^\mathrm{nr}_{T\setminus S}\right)\rlap{\ .}
    \end{align*}
    (2) If $S=\emptyset$, $X\setminus Y\subseteq T$, and the order of $\mathcal{M}$ is invertible over $Y_T$, then the following equality holds in $\mathrm{Set}(\mathscr{F}(X),\mathbb{Q}/\mathbb{Z})$:
    \begin{align*}
        \mathrm{BF}_{X_T}^{\xi_T}|_{\mathscr{F}(X_S)}=\mathrm{BF}_X+\left(\mathrm{BF}_T^{\mathrm{nr},\xi_T}\circ\partial_T^\mathrm{nr.}\right)\rlap{\ .}
    \end{align*}
\end{quote}
\end{prop}
\begin{proof} Follows from taking the trivialization of Theorem \ref{Decomp-formula}.
\end{proof}

\section{Quantum arithmetic BF theory}

\subsection{Hermitian line bundles}

We have an obvious homomorphism:
\begin{align}
    \xymatrix{\mathbb{Q}/\mathbb{Z} \ar[r] & \mathbb{C}^\times & t \ar@{|->}[r] & \exp\left(2\pi\sqrt{-1}t\right)}
\end{align}
whose image lies in $\mathrm{U}(1)$. Using this, we form
\begin{align*}
    L_S:=\mathscr{L}_S\times_{\mathbb{Q}/\mathbb{Z}}\mathbb{C}
\end{align*}
so that we have a canonical map on $\mathscr{F}_S$:
\begin{align*}
    \xymatrix{\eta_S:\mathscr{L}_S \ar[r] & L_S & \lambda \ar@{|->}[r] & [\lambda,1]\rlap{\ .}}
\end{align*}
Then $\varpi_S:\mathscr{L}_S\rightarrow\mathscr{F}_S$ induces a complex line bundle:
\begin{align*}
    \xymatrix{\varpi_{S,\mathbb{C}}:L_S \ar[r] & \mathscr{F}_S\rlap{\ .}}
\end{align*}
Equip the standard inner product on each fiber $\mathbb{C}$. Since $\mathscr{F}_S$ is discrete, $\varpi_{S,\mathbb{C}}$ becomes a Hermitian line bundle. We use the Haar measure on the locally compact abelian group $\mathscr{F}_S$ normalized so that every point has measure $1$. Since $\mathscr{F}_S$ is finite by Section \ref{Localthy-torsor}, the resulting Hermitian inner product is well-defined:
\begin{align*}
    \xymatrix{\langle\cdot,\cdot\rangle:\Gamma(\mathscr{F}_S,L_S)\times\Gamma(\mathscr{F}_S,L_S) & (\sigma,\tau) \ar@{|->}[r] & \displaystyle\sum_{\alpha\in\mathscr{F}_S}\sigma(\alpha)\overline{\tau(\alpha)}}.
\end{align*}
We take the Hilbert space:
\begin{align*}
    \mathscr{H}_S:=(\Gamma(\mathscr{F}_S,L_S),\langle\cdot,\cdot\rangle)\rlap{\ .}
\end{align*}

Given $\xi\in\Gamma(\mathscr{F}_S,\mathscr{L}_S)$, taking the associated line bundles of the isomorphism \eqref{scrLS-trivialization}, we get a trivialization of a complex line bundle:
\begin{align*}
    \xymatrixcolsep{1.5pc}\xymatrix{\Phi_{S,\mathbb{C}}^\xi:L_S \ar[r]^-\sim & \mathscr{F}_S\times\mathbb{C} & (\lambda,z) \ar@{|->}[r] & \displaystyle\left(\varpi_S(\lambda),z\cdot\exp\left(2\pi\sqrt{-1}\mathrm{pr}_{\mathbb{Q}/\mathbb{Z}}\Phi_S^\xi(\lambda)\right)\right)}
\end{align*}
which induces an isomorphism of Hilbert spaces:
\begin{align}\label{HS-trivialization}
    \xymatrix{\mathscr{H}_S \ar[r]^-\sim & \mathrm{Set}(\mathscr{F}_S,\mathbb{C})\rlap{\ .}}
\end{align}
where the right hand side is endowed with the standard Hermitian inner product.

\begin{rmk}\label{HS-tensor-decomp} External tensor products of $(L_u)_{u\in S}$ along $\mathrm{pr}_u:\mathscr{F}_S\rightarrow\mathscr{F}_u$ can be described as follows:
\begin{align*}
    \bigotimes_{u\in S}\mathrm{pr}_u^\ast L_u=\left\{\left(\rho,\bigotimes_{u\in S}[\lambda_u,z_u]\right)\ \middle|\ \begin{array}{c}
        \rho=(\rho_u)_{u\in S}\in\mathscr{F}_S  \\\relax
        [\lambda_u,z_u]\in L_u(\rho_u)
    \end{array} \right\}\rlap{\ .}
\end{align*}
Then we have a well-defined isomorphism of complex line bundles over $\mathscr{F}_S$:
\begin{align*}
    \xymatrix{\displaystyle\bigotimes_{u\in S}\mathrm{pr}_u^\ast L_u \ar[r] & L_S & \displaystyle\left(\rho,\bigotimes_{u\in S}[\lambda_u,z_u]\right) \ar@{|->}[r] & \displaystyle\left[[\lambda_u]_{u\in S},\prod_{u\in S}z_u\right]}
\end{align*}
with an inverse:
\begin{align*}
    \xymatrix{L_S \ar[r] & \displaystyle\bigotimes_{u\in S}\mathrm{pr}_u^\ast L_u & [\lambda,z] \ar@{|->}[r] & \displaystyle\left(\varpi_S(\lambda),z\cdot\bigotimes_{u\in S}[\lambda_u,1]\right)}\rlap{\ .}
\end{align*}
This gives an isomorphism of Hilbert spaces:
\begin{align*}
    \xymatrix{\displaystyle\mathop{\widehat{\bigotimes}}_{u\in S}\mathscr{H}_u \ar[r]^-\sim & \mathscr{H}_S & \displaystyle\bigotimes_{u\in S}\sigma_u \ar@{|->}[r] & \displaystyle\bigotimes_{u\in S}\mathrm{pr}_u^\ast\sigma_u\rlap{\ .}}
\end{align*}
Given sections $\xi_u\in\Gamma(\mathscr{F}_u,\mathscr{L}_u)$ for each $u\in S$, if we take
\begin{align*}
    \xi_S:=\mathop{\boxplus}_{u\in S}\xi_u\in\Gamma(\mathscr{F}_S,\mathscr{L}_S)
\end{align*}
then the above isomorphism fits into the following commutative square:
\begin{align*}
    \xymatrixrowsep{3pc}\xymatrix{
    \displaystyle\mathop{\widehat{\bigotimes}}_{u\in S}\mathscr{H}_u \ar[d]^-\wr \ar[r]^-\sim & \mathscr{H}_S \ar[d]^-\wr & & \\
    \displaystyle\bigotimes_{u\in S}\Gamma(\mathscr{F}_u,\mathbb{C}) \ar[r]^-\sim & \Gamma(\mathscr{F}_S,\mathbb{C}) & \displaystyle\bigotimes_{u\in S}f_u \ar@{|->}[r] & \displaystyle\prod_{u\in S}(f_u\circ\mathrm{pr}_u)
    }
\end{align*}
where the columns come from \eqref{HS-trivialization} and the bottom row is the canonical isomorphism.
\end{rmk}

\begin{notation}\label{HS-dual} Denote $L_{S^\ast}:=\mathscr{L}_{S^\ast}\times_{\mathbb{Q}/\mathbb{Z}}\mathbb{C}$ the complex line bundle associated to $\mathscr{L}_{S^\ast}$. Then the fiberwise complex conjugation gives a $\mathbb{C}$-semilinear isomorphism:
\begin{align*}
    \xymatrix{L_S \ar[r]^-\sim & L_{S^\ast} & (\lambda,z) \ar@{|->}[r] & (\lambda,\overline{z})\rlap{\ .}}
\end{align*}
This gives a topological $\mathbb{C}$-semilinear isomorphism:
\begin{align*}
    \xymatrix{\mathscr{H}_S \ar[r]^-\sim & \mathscr{H}_{S^\ast} & \sigma \ar@{|->}[r] & \overline{\sigma}\rlap{\ .}}
\end{align*}
Consequently, the pairing
\begin{align*}
    \xymatrix{\mathscr{H}_{S^\ast}\times\mathscr{H}_S \ar[r] & \mathbb{C} & (\chi,\theta) \ar@{|->}[r] & \displaystyle\sum_{\rho\in\mathscr{F}_S}\chi(\rho)\theta(\rho)}
\end{align*}
recovers the defining Hermitian inner product on $\mathscr{H}_S$. Therefore, the above isomorphism gives an isomorphism of Hilbert spaces:
\begin{align*}
    \xymatrix{\mathscr{H}_{S^\ast} \ar[r]^-\sim & \mathscr{H}_S^\ast & \chi \ar@{|->}[r] & \displaystyle\left(\theta\longmapsto\sum_{\rho\in\mathscr{F}_S}\chi(\rho)\theta(\rho)\right)\rlap{\ .}}
\end{align*}
In what follows, we freely use this identification.
\end{notation}

\subsection{Partition functions} We first study a possible finiteness condition on $\mathscr{F}(X_S)$.

\begin{defn}\label{defn-Selmer} Let $G$ be a commutative finite $K$-group scheme. By \cite[1.11]{Kes15}, $H_\mathrm{fppf}^1(K_x,G)$ is a locally compact Hausdorff abelian topological group which is discrete if $G$ is \'etale over $K_x$. A \emph{Selmer condition} for a commutative finite $K$-group scheme $G$ is a compact subgroup of the form:
\begin{align*}
    W=\prod_{x\in X^\mathrm{cl}}W_x\subseteq\prod_{x\in X^\mathrm{cl}}H_\mathrm{fppf}^1(K_x,G)
\end{align*}
which are endowed with the product topology, such that there is a nonempty open subset $U\subseteq X$ and a commutative finite flat $U$-model $\mathcal{G}$ of $G$ satisfying the condition that $W_u\subseteq H_\mathrm{fppf}^1(\mathcal{O}_u,\mathcal{G})$ for every $u\in U^\mathrm{cl}$. Then the associated \emph{Selmer group} is defined to be the pullback
\begin{align*}
    \xymatrix{
    \mathrm{Sel}(G,W) \ar[d] \ar@{^(->}[r] & H_\mathrm{fppf}^1(K,G) \ar[d] \\
    W \ar@{^(->}[r] & \displaystyle\prod_{x\in X^\mathrm{cl}}H_\mathrm{fppf}^1(K_x,G)
    }
\end{align*}
of topological abelian groups, where $H_\mathrm{fppf}^1(K,G)$ is regarded as a discrete group.
\end{defn}

\begin{rmk}\label{Sel-finite} By \cite[Theorem 3.2]{Kes17}, $\mathrm{Sel}(G,W)$ is finite.
\end{rmk}

\begin{ex} Let $\mathcal{G}$ be a nice commutative group scheme over $U$. By Remark \ref{Hfppftop} and the proof of Lemma \ref{FX}, $H_\mathrm{fppf}^1(U,\mathcal{G})$ and $D_\mathrm{fppf}^1(U,\mathcal{G})$ are Selmer groups.
\end{ex}

\begin{rmk}\label{cptbdycond} By Section \ref{Localthy-torsor}, each $\mathrm{BC}_x$ is finite for every $x\in X^\mathrm{cl}$. Denote
\begin{align*}
    \xymatrix{\mathrm{pr}_1:\mathscr{F}_x \ar[r] & H^1(G_{K,x},M_1^\vee) & \mathrm{pr}_2:\mathscr{F}_x \ar[r] & H^1(G_{K,x},M_2)}
\end{align*}
the projection onto each component. Since $\mathrm{BC}_x$ is compact, $\mathrm{pr}_i(\mathrm{BC}_x)$ is compact for $i=1,2$ so we may define the corresponding Selmer groups:
\begin{align*}
    \mathrm{Sel}\left(M_1^\vee,\prod_{x\in X^\mathrm{cl}}\mathrm{pr}_i(\mathrm{BC}_x)\right),\quad\mathrm{Sel}\left(M_2,\prod_{x\in X^\mathrm{cl}}\mathrm{pr}_2(\mathrm{BC}_x)\right)\rlap{\ .}
\end{align*}
By construction, we have the following pullback square:
\begin{align*}
    \xymatrix{
    \mathscr{F}(X_S) \ar[d] \ar[r] & \displaystyle\mathrm{Sel}\left(M_1^\vee,\prod_{x\in X^\mathrm{cl}}\mathrm{pr}_1(\mathrm{BC}_x)\right)\times\mathrm{Sel}\left(M_2,\prod_{x\in X^\mathrm{cl}}\mathrm{pr}_2(\mathrm{BC}_x)\right) \ar[d] \\
    \mathrm{BC}(\partial X_S) \ar[r] & \displaystyle\prod_{x\in X^\mathrm{cl}}(\mathrm{pr}_1(\mathrm{BC}_x)\times\mathrm{pr}_2(\mathrm{BC}_x))
    }
\end{align*}
where the bottom row is the product of the following composite of canonical maps:
\begin{align*}
    \xymatrix{\mathrm{BC}(\partial X_S)_x \ar[r] & \mathrm{BC}_x \ar@{^(->}[r] & \mathrm{pr}_1(\mathrm{BC}_x)\times\mathrm{pr}_2(\mathrm{BC}_x)\rlap{\ .}}
\end{align*}
This map is injective for all but finitely many $x\in X^\mathrm{cl}$. Since $\mathrm{BC}(\partial X_S)_x$ is finite for every $x\in X^\mathrm{cl}$, the bottom row of the above pullback square has finite fibers. Since the Selmer groups are finite by Remark \ref{Sel-finite}, we conclude that $\mathscr{F}(X_S)$ is finite.
\end{rmk}

\begin{defn}\label{partftn-X} We define
\begin{align*}
    Z_X:=\sum_{\rho\in\mathscr{F}(X)}\exp\left(2\pi\sqrt{-1}\mathrm{BF}_X(\rho)\right)\in\mathbb{C}\rlap{\ .}
\end{align*}
Since $\mathscr{F}(X)$ is finite by Remark \ref{cptbdycond}, this sum is finite.
\end{defn}

\begin{defn}\label{partftn-XS} Given a finite subset $S\subseteq X^\mathrm{cl}$, denote each fiber as in the following left pullback square so that $BF_{X_S}$ restricts to each fiber and the following right diagram becomes a pullback square:
\begin{align*}
    \xymatrix{
    \mathscr{F}(X_S,\rho_S) \ar[d] \ar@{^(->}[r] & \mathscr{F}(X_S) \ar[d]^-{\partial_S} \\
    \{\rho_S\} \ar@{^(->}[r] & \mathscr{F}_S
    }
    \quad\quad
    \xymatrix{
    \mathscr{F}(X_S,\rho_S) \ar@{-->}[d]_-{\mathrm{BF}_{X_S}} \ar@{^(->}[r] & \mathscr{F}(X_S) \ar[d]^-{\mathrm{BF}_{X_S}} \\
    \mathscr{L}_S(\rho_S) \ar@{^(->}[r] & \mathscr{L}_S\rlap{\ .}
    }
\end{align*}
Then we define the arithmetic BF partition functions on $X_S$ as follows:
\begin{align*}
    \xymatrix{Z_{X_S}:\mathscr{F}_S \ar[r] & L_S & \rho_S \ar@{|->}[r] & \displaystyle\sum_{\rho\in\mathscr{F}(X_S,\rho_S)}\left(\eta_S\circ\mathrm{BF}_{X_S}\right)(\rho)\rlap{\ .}}
\end{align*}
By the above observation, each summand lies in the fiber of $\rho_S$ so the sum makes sense. By Remark \ref{cptbdycond}, the sum is finite.
\end{defn}

\begin{rmk}\label{partftn-compatibility} As a sanity check, consider the case $S=\emptyset$ in Definition \ref{partftn-XS}. Since $\mathscr{F}_\emptyset=\ast$ and $\mathscr{L}_\emptyset=\mathbb{Q}/\mathbb{Z}$, the associated Hermitian line bundle becomes $\mathbb{C}\rightarrow\ast$ together with
\begin{align*}
    \xymatrix{\eta_\emptyset:\mathbb{Q}/\mathbb{Z} \ar[r] & \mathbb{C} & \lambda \ar@{|->}[r] & \exp\left(2\pi\sqrt{-1}\lambda\right)}
\end{align*}
Combining these, we get
\begin{align*}
    \xymatrix{Z_{X_\emptyset}:\ast \ar[r] & \mathbb{C} & \ast \ar@{|->}[r] & \displaystyle\sum_{\rho\in\mathscr{F}(X)}\exp\left(2\pi\sqrt{-1}\mathrm{BF}_X(\rho)\right)}
\end{align*}
which recovers Definition \ref{partftn-X}.
\end{rmk}

\begin{defn} Given finite subsets $S\subsetneq T\subseteq X^\mathrm{cl}$, denote each fiber as in the following left pullback squares so that $\mathrm{BF}_{T\setminus S}^\mathrm{nr}$ restricts to each fiber and the following right diagram becomes a pullback square:
\begin{align*}
    \xymatrix{
    \mathrm{BC}(\partial X_S,\rho_{T\setminus S}) \ar@{^(->}[d] \ar@{^(->}[r] & \mathrm{BC}(\partial X_S)_{T\setminus S} \ar@{^(->}[d] \\
    \mathscr{F}_{T\setminus S}^\mathrm{nr}(\rho_{T\setminus S}) \ar[d] \ar@{^(->}[r] & \mathscr{F}_{T\setminus S}^\mathrm{nr} \ar[d] \\
    \{\rho_{T\setminus S}\} \ar@{^(->}[r] & \mathscr{F}_{T\setminus S}
    }
    \quad\quad
    \xymatrix{
    \mathrm{BC}(\partial X_S,\rho_{T\setminus S}) \ar@{^(->}[d] \ar@{^(->}[r] & \mathrm{BC}(\partial x_S)_{T\setminus S} \ar@{^(->}[d] \\
    \mathscr{F}_{T\setminus S}^\mathrm{nr}(\rho_{T\setminus S}) \ar@{-->}[d]_-{\mathrm{BF}_{T\setminus S}^\mathrm{nr}} \ar@{^(->}[r] & \mathscr{F}_{T\setminus S}^\mathrm{nr} \ar[d]^-{\mathrm{BF}_{T\setminus S}^\mathrm{nr}} \\
    \mathscr{L}_{T\setminus S}(\rho_{T\setminus S}) \ar@{^(->}[r] & \mathscr{L}_{T\setminus S}\rlap{\ .}
    }
\end{align*}
Then we define
\begin{align*}
    \xymatrix{Z_{T\setminus S}^{\partial X_S}:\mathscr{F}_{T\setminus S} \ar[r] & L_{T\setminus S} & \rho_{T\setminus S} \ar@{|->}[r] & \displaystyle\sum_{\rho\in\mathrm{BC}(\partial X_S,\rho_{T\setminus S})}\left(\eta_{T\setminus S}\circ\mathrm{BF}_{T\setminus S}^\mathrm{nr}\right)(\rho)\rlap{\ .}}
\end{align*}
By the above observation, each summand lies in the fiber of $\rho_{T\setminus S}$ so the sum makes sense. Since $\mathscr{F}_{T\setminus S}^\mathrm{nr}$ is finite by Section \ref{Localthy-BFnr}, the sum is finite.
\end{defn}

\begin{rmk} Along the isomorphism \eqref{HS-trivialization}, we get functions
\begin{align*}
    \xymatrix{Z_{X_S}^\xi:\mathscr{F}_S \ar[r] & \mathbb{C} & \rho_S \ar@{|->}[r] & \displaystyle\sum_{\rho\in\mathscr{F}(X_S,\rho_S)}\exp\left(2\pi\sqrt{-1}\mathrm{BF}_{X_S}^\xi(\rho)\right)}
\end{align*}
\begin{align*}
    \xymatrix{Z_{T\setminus S}^{\partial X_S,\xi}:\mathscr{F}_{T\setminus S} \ar[r] & \mathbb{C} & \rho_{T\setminus S} \ar@{|->}[r] & \displaystyle\sum_{\rho\in\mathrm{BC}(\partial X_S,\rho_{T\setminus S})}\exp\left(2\pi\sqrt{-1}\mathrm{BF}_S^{\mathrm{nr},\xi}(\rho)\right)}
\end{align*}
corresponding to the arithmetic BF partition functions.
\end{rmk}

\subsection{Gluing formula} Let $S\subseteq T\subseteq Y^\mathrm{cl}$ be finite subsets. By Remark \ref{HS-tensor-decomp}, we have a canonical isomorphism of Hilbert spaces:
\begin{align*}
    \xymatrix{\mathscr{H}_S\mathbin{\widehat{\otimes}}\mathscr{H}_{T\setminus S} \ar[r]^-\sim & \mathscr{H}_T\rlap{\ .}}
\end{align*}
Using this, we get the following 
\begin{align*}
    \xymatrix{\langle\cdot,\cdot\rangle:\mathscr{H}_T\mathbin{\widehat{\otimes}}\mathscr{H}_{(T\setminus S)^\ast} \ar[r]^-\sim & \mathscr{H}_S\mathbin{\widehat{\otimes}}\mathscr{H}_{T\setminus S}\mathbin{\widehat{\otimes}}\mathscr{H}_{T\setminus S}^\ast \ar[r]^-{1\mathbin{\widehat{\otimes}}\mathrm{ev}} & \mathscr{H}_S\mathbin{\widehat{\otimes}}\mathbb{C} \ar[r]^-\sim & \mathscr{H}_S}
\end{align*}
where $\mathrm{ev}$ is the evaluation pairing. Given $\tau\in\mathscr{H}_T$ and $\chi\in\mathscr{H}_{(T\setminus S)^\ast}$, we have
\begin{align*}
    \xymatrix{\langle\tau,\chi\rangle:\mathscr{F}_S \ar[r] & L_S & \alpha \ar@{|->}[r] & \displaystyle\sum_{\gamma\in\mathscr{F}_{T\setminus S}}\tau(\alpha\oplus\gamma)\chi(\gamma)\rlap{\ .}}
\end{align*}

\begin{thm}[Gluing formula] Let $S\subseteq T\subseteq Y^\mathrm{cl}$ be finite subsets.
\begin{quote}
    (1) If $S\neq\emptyset$, then the following equality holds in $\mathscr{H}_S$:
    \begin{align*}
        Z_{X_S}=\left\langle Z_{X_T},Z_{(T\setminus S)^\ast}^{\partial X_S}\right\rangle\rlap{\ .}
    \end{align*}
    (2) If $S=\emptyset$, $X\setminus Y\subseteq T$, and the order of $\mathcal{M}$ is invertible on $Y_T$, then the following equality holds in $\mathbb{C}$:
    \begin{align*}
        Z_X=\left\langle Z_{X_T},Z_{T^\ast}^{\partial X}\right\rangle\rlap{\ .}
    \end{align*}
\end{quote}
\end{thm}
\begin{proof} (1) Given $\alpha\in\mathscr{F}_S$, we have the following commutative cube all of whose faces are pullback squares:
\begin{align*}
    \xymatrix@!0@R=4pc@C=6pc{
    & \displaystyle\coprod_{\gamma\in\mathscr{F}_{T\setminus S}}\mathscr{F}(X_T,\alpha\oplus\gamma)\times\mathrm{BC}(\partial X_S,\gamma) \ar[dl] \ar[dd]|\hole \ar@{^(->}[rr] & & \mathscr{F}(X_S) \ar[dl] \ar[dd]^(.25){(\partial_S,\partial_{T\setminus S}^\mathrm{nr})} \\
    \displaystyle\coprod_{\gamma\in\mathscr{F}_{T\setminus S}}\mathscr{F}(X_T,\alpha\oplus\gamma) \ar[dd] \ar@{^(->}[rr] & & \mathscr{F}(X_T) \ar[dd]^(.25){\partial_T} & \\
    & \{\alpha\}\times\mathrm{BC}(\partial X_S)_{T\setminus S} \ar[dl] \ar@{^(->}[rr]|\hole & & \mathscr{F}_S\times\mathrm{BC}(\partial X_S)_{T\setminus S}\rlap{\ .} \ar[dl] \\
        \{\alpha\}\times\mathscr{F}_{T\setminus S} \ar@{^(->}[rr] & & \mathscr{F}_T &
    }
\end{align*}
Since $\{\alpha\}\times\mathrm{BC}(\partial X_S)_{T\setminus S}=\mathrm{pr}_S^{-1}(\alpha)$, we conclude that 
\begin{align*}
    \mathscr{F}(X_S,\alpha)\cong\coprod_{\gamma\in\mathscr{F}_{T\setminus S}}\mathscr{F}(X_T,\alpha\oplus\gamma)\times\mathrm{BC}(\partial X_S,\gamma)\rlap{\ .}
\end{align*}
Using this bijection and Proposition \ref{Decomp-formula-with-section}, we compute
\begin{align*}
    &\quad\left\langle Z_{X_T},Z_{(T\setminus S)^\ast}^{\partial X_S}\right\rangle(\alpha)\\
    &=\sum_{\gamma\in\mathscr{F}_{T\setminus S}}Z_{X_T}(\alpha\oplus\gamma)\otimes Z_{(T\setminus S)^\ast}^\mathrm{nr}(\gamma)\\
    &=\sum_{\gamma\in\mathscr{F}_{T\setminus S}}\sum_{\substack{\widetilde{\beta}\in\mathscr{F}(X_T,\alpha\oplus\gamma) \\ \widetilde{\gamma}\in\mathrm{BC}(\partial X_S,\gamma)}}\left(\eta_T\circ\mathrm{BF}_{X_T}\right)(\widetilde{\beta})\otimes\left(\eta_{T\setminus S}\circ\mathrm{BF}_{(T\setminus S)^\ast}^{\partial X_S}\right)(\widetilde{\gamma})\\
    &=\sum_{\gamma\in\mathscr{F}_{T\setminus S}}\sum_{\substack{\widetilde{\beta}\in\mathscr{F}(X_T,\alpha\oplus\gamma) \\ \widetilde{\gamma}\in\mathrm{BC}(\partial X_S,\gamma)}}\exp\left(2\pi\sqrt{-1}\left(\mathrm{BF}_{X_T}^{\xi_S\boxplus\xi_{T\setminus S}}(\widetilde{\beta})-\mathrm{BF}_{T\setminus S}^{\mathrm{nr},\xi_{T\setminus S}}(\widetilde{\gamma})\right)\right)\\
    &=\sum_{\widetilde{\alpha}\in\mathscr{F}(X_S,\alpha)}\exp\left(2\pi\sqrt{-1}\mathrm{BF}_{X_S}^{\xi_S}(\widetilde{\alpha})\right)\\
    &=\sum_{\widetilde{\alpha}\in\mathscr{F}(X_S,\alpha)}\left(\eta_S\circ\mathrm{BF}_{X_S}\right)(\widetilde{\alpha})\\
    &=Z_{X_S}(\alpha)\rlap{\ .}
\end{align*}
Therefore, the first assertion follows.\\
\indent (2) With Remark \ref{partftn-compatibility}, the same argument as above works.
\end{proof}

\section{The Cassels-Tate pairing for finite Galois modules}
We briefly review the Cassels-Tate pairing \cite[Definition 2.2]{MS24}.
First, we recall the definition of the category $\mathrm{SMod}_K$ of ``Selmerable'' modules defined in \cite[section 2.2.1]{MS24}.
An object in $\mathrm{SMod}_K$ is a pair $(M, \mathscr{W})$, where $M$ is a finite $G_K$-module of order indivisible by the characteristic of $K$ and $\mathscr{W}$ is a Selmer condition for $M$. Specifically, $\mathscr{W}$ is a compact open subgroup of the restricted direct product of local cohomology groups with respect to unramified cohomology classes
$$\mathscr{W} \subseteq \sideset{}{'}\prod_{v} H^1(G_{K,v}, M);$$
for almost all places $v$, $\mathscr{W}_v$ coincides with the unramified subgroup 
$H^1(G_{K,v}^{\mathrm{nr}}, M^{I_{K,x}})$.
A morphism between two objects $(M_1, \mathscr{W}_1) \to (M_2, \mathscr{W}_2)$ is a $G_K$-equivariant map $\phi: M_1 \to M_2$ such that the induced map on local cohomology respects the Selmer conditions
$$\phi_*(\mathscr{W}_1) \subseteq \mathscr{W}_2.$$
where $\phi_*$ is the map $H^1(G_{K,v}, M_1) \to H^1(G_{K,v}, M_2)$ for every place $v$.
This category is a quasi-abelian category (\cite[Defintion 4.10]{MS21}) and there is a notion of an exact sequence. The ``Selmerable'' part of the name refers to the fact that these modules are compatible with forming Selmer groups. For a sequence
$$0 \to (M_1, \mathscr{W}_1) \xrightarrow{\iota} (M, \mathscr{W}) \xrightarrow{\pi} (M_2, \mathscr{W}_2) \to 0$$
to be exact in $\mathrm{SMod}_K$, the underlying sequence of $G_K$-modules must be exact, and the local conditions must satisfy the Selmer-exactness conditions:
$$\iota^{-1}(\mathscr{W}) = \mathscr{W}_1, \quad \pi(\mathscr{W}) = \mathscr{W}_2.$$

For example, consider an $n$-isogeny of elliptic curves $\phi: E \to E'$ over $K$. This gives a sequence of $G_K$-modules:
$$0 \to E[\phi] \to E[n\phi] \to E[n] \to 0.$$
We can define $\mathscr{W}=\prod_v' \mathscr{W}_v$ using the local points of $E$:
$\mathscr{W}_v = \mathrm{im}(E(K_v) \to H^1(G_{K,v}, E[n\phi]))$. The conditions $\iota^{-1}(\mathscr{W}) = \mathscr{W}_1$ and $\pi(\mathscr{W}) = \mathscr{W}_2$ are naturally satisfied if $\mathscr{W}_1$ is the local images of the Kummer map for $E[\phi]$ (via $E'(K_v)$), and $\mathscr{W}_2$ is the local images for $E[n]$.

The input for the Cassels-Tate pairing is an exact sequence of finite $G_K$-modules:
    \begin{align}\label{CTexact}
    \xymatrix{0 \ar[r] & M_1 \ar[r]^-\iota & M \ar[r]^-\pi & M_2 \ar[r] & 0\rlap{\ .}}
    \end{align}
Let $\mathscr{W}$ be a compact subgroup of the restricted direct product $\prod'_v H^1(G_{K,v}, M)$ with respect to the subgroup $H^1(G_{K,v}^{\mathrm{nr}}, M^{I_{K,x}}) \hookrightarrow H^1(G_{K,v}, M)$.
 Let $\mathscr{W}^\perp$ be the orthogonal complement of $\mathscr{W}$ with respect to the pairing induced by the Tate local duality.
 Motivated by the notion of the exact sequence in $\mathrm{SMod}_K$, we impose the following condition on \eqref{CTexact}:
\begin{eqnarray}\label{SelmerM}
    \iota^{-1}(\mathscr{W}) =\mathscr{W}_1 \text{ and } \pi(\mathscr{W})=\mathscr{W}_2.
\end{eqnarray}
We denote by $E$ the corresponding exact sequence in the category $\mathrm{SMod}_K$ defined in \cite[section 2.2.1]{MS24}:
\begin{align}\label{Eexact}
    E:=\left[\xymatrix{0 \ar[r] & (M_1,\mathscr{W}_1) \ar[r]^-\iota & (M,\mathscr{W}) \ar[r]^-\pi & (M_2,\mathscr{W}_2) \ar[r] & 0}\right].
\end{align}
Then we have the dual exact sequence of $G_K$-modules
\begin{align*}
    \xymatrix{0 \ar[r] & M_2^\vee \ar[r]^-{\pi^\vee} & M^\vee \ar[r]^-{\iota^\vee} & M_1^\vee \ar[r] & 0}
\end{align*}
such that ${(\pi^\vee)}^{-1}(\mathscr{W}^\perp) = \mathscr{W}_2^{\perp}$ and $\iota^\vee (\mathscr{W}^\perp) = \mathscr{W}_1^\perp$.
Define
\begin{align*}
    \mathrm{Sel}(M,\mathscr{W}):=\left(\xymatrix{H^1(G_K,M) \ar[r] & \displaystyle\sideset{}{'}\prod_{v\in X^\mathrm{cl}}H^1(G_{K,v},M)/\mathscr{W}}\right)
\end{align*}

With the input data \eqref{CTexact} and $\eqref{SelmerM}$ (equivalently, $\eqref{Eexact}$), Morgan-Smith attached a version of the Cassels-Tate pairing. 
\begin{thm}\cite[Theorem 2.1]{MS24}
    There is a bilinear pairing, called the Cassels-Tate pairing,
\begin{align*}
    \xymatrix{\mathrm{CTP}_E:   \mathrm{Sel} (M_1^\vee, \mathscr{W}_1^\perp) \times\mathrm{Sel} (M_2,\mathscr{W}_2)\ar[r] & \mathbb{Q}/\mathbb{Z}\rlap{\ .}}
\end{align*}
with left kernel $\pi(\mathrm{Sel} (M, \mathscr{W}))$ and right kernel $\iota^\vee(\mathrm{Sel} (M^\vee, \mathscr{W}^{\perp}))$.
\end{thm}

To define the Cassels-Tate pairing $\mathrm{CTP}(\rho_2, \rho_1)$, Morgan-Smith make `global' choices consisting of 
\begin{enumerate}
    \item cocycles $(\alpha_1, \alpha_2)$ representing $(\rho_1, \rho_2)$,
    \item a cochain $f \in C^1(G_K, M)$ satisfying $\alpha_2 = \pi \circ f$, and
    \item a cochain $\epsilon \in C^2(G_K, (K^{\mathrm{sep}})^\times)$ such that
    \begin{align*}
        d\epsilon =\alpha_1  \cup d f =  \alpha_1 \cup \iota^{-1}(df).
    \end{align*}
\end{enumerate}
They further make `local' choices consisting of
\begin{enumerate}
    \item cocycles $\alpha_{v,2,M}$ in $Z^1(G_{K,v}, M)$ for each place $v$ of $K$ such that $\pi(\alpha_{v,2,M})=\alpha_{v,2}$ and $(\alpha_{v,2,M})_v$ represent a class in $\mathscr{W}$.
\end{enumerate}
Then
\begin{align}\label{gv}
        \gamma_v :=   \partial_v\epsilon- \partial_v\alpha_{1} \cup (\alpha_{v,2,M}-\sigma \circ \partial_v\alpha_{2} )  ,
\end{align}
where $\sigma\in\mathrm{Set}(M_2,M)$ such that $ \pi\circ\sigma=\mathrm{Id}_{M_2}$, lies in $Z^2(G_{K,v}, (K^{\mathrm{sep}})^\times)$ for each $\nu \in X^{\mathrm{cl}}$, and 
\begin{align*}
    \mathrm{CTP}_E(\rho_1, \rho_2) := \sum_{v \in X^{\mathrm{cl}}} \mathrm{inv}_v (\gamma_v).
\end{align*}
This definition is independent of these choices: it is independent of  $\epsilon$ such that $d\epsilon =\alpha_1  \cup d f$; it is independent of the choice of cocycles $(\alpha_1, \alpha_2)$ that represent the cohomology classes $(\rho_1, \rho_2)$; 
it is independent of the choice of $f$ such that $\pi(f) = \alpha_2$;
it is independent of the choice of $\alpha_{v,2,M} \in Z^1(G_{K,v}, M)$ such that $\pi(\alpha_{v,2,M}) = \alpha_{v,2}$ and its class is in $\mathscr{W}_v$;
it is independent of the set-theoretic lift $\sigma$ such that $\pi \circ \sigma =Id_{M_2}$.
This independence of the Cassels-Tate pairing from its auxiliary choices relies on the global reciprocity law, which states that for any global class $c \in H^2(G_K, \mathbb{G}_m) \cong \mathrm{Br}(K)$, the sum of its local invariants is zero:
$$\sum_{v \in X^{\mathrm{cl}}} \mathrm{inv}_v(\partial_v c) = 0.$$

Recall that $Y\subseteq X$ is an open subset and that we have the following input data of the arithmetic BF theory. 
\begin{itemize}
    \item An exact sequence \eqref{BF-input} of nice commutative group schemes over $Y$:
    \begin{align*}
        \xymatrix{0 \ar[r] & \mathcal{M}_1 \ar[r] & \mathcal{M} \ar[r] & \mathcal{M}_2 \ar[r] & 0\rlap{\ .}}
    \end{align*}
    \item A subgroup of ``boundary conditions''
    \begin{align*}
        \mathrm{BC}=\prod_{x\in X^\mathrm{cl}}\mathrm{BC}_x\subseteq\prod_{x\in X^\mathrm{cl}}\mathscr{F}_x:=\prod_{x\in X^\mathrm{cl}}H^1(G_{K,x},M_1^\vee)\times H^1(G_{K,x},M_2)
    \end{align*}
    such that $\im(\mathscr{F}_y^\mathrm{nr}\rightarrow\mathscr{F}_y)\subseteq\mathrm{BC}_y\subseteq\mathscr{F}_y$ for $y\in Y^\mathrm{cl}$.
\end{itemize}
For a comparison of the arithmetic BF functional with the Cassels-Tate pairing, we assume that, after evaluating $K^\mathrm{sep}$ to \eqref{BF-input}, we get the same exact sequence as \eqref{CTexact}:
    \begin{align*}
    \xymatrix{0 \ar[r] & M_1 \ar[r]^-\iota & M \ar[r]^-\pi & M_2 \ar[r] & 0\rlap{\ ,}}
    \end{align*}
and $\displaystyle\prod_{x\in X^\mathrm{cl}}\mathrm{BC}_x=\iota^{-1}(\mathscr{W}) \times\pi(\mathscr{W}).$
Under this assumption, we would like to compare the arithmetic BF theory on $\mathscr{F}(X_S)$ (defined in \eqref{FXS-defn}) and the Cassels-Tate pairing.

\begin{comment}
and there exists a subgroup of ``boundary conditions''
\begin{eqnarray}\label{BCM}
\mathrm{BC}_M=\prod_{x \in X^{\mathrm{cl}}} \mathrm{BC}_{M,x} \subseteq \prod_{x \in X^{\mathrm{cl}}} H^1(G_{K,x}, M)
\end{eqnarray}
such that 
\begin{eqnarray}
    \iota^{-1}(\mathrm{BC}_M) \times \pi(\mathrm{BC}_M) \simeq \mathrm{BC}.
\end{eqnarray}
$H^1(G_{K,y}^{\mathrm{nr}}, M^{I_{K,y}}) \subseteq \mathrm{BC}_{M,y}$ for $y \in Y^{\mathrm{cl}}$ and
\begin{eqnarray*}
    \mathrm{BC}_x =\iota^{\vee}(\mathrm{BC}_{M,x}^{\perp}) \times  \pi (\mathrm{BC}_{M,x})\subseteq \mathscr{F}_x:=H^1(G_{K,x},M_1^\vee)\times H^1(G_{K,x},M_2)
\end{eqnarray*}
for all $x \in X^{\mathrm{cl}}.$ (Denote by $\mathrm{BC}_{M,x}^{\perp}$ the orthogonal complement of $\mathrm{BC}_{M,x}$ in $H^1(G_{K,x}, M)$ under the local Tate pairing.)
\end{comment}


\subsection{Global arithmetic BF functional and the Cassels-Tate pairing}
We assume that $S=\emptyset$ and so 
\begin{align*}
    \mathrm{BC}(\partial X_\emptyset):=\prod_{x\in X\setminus Y}\{0\}\times\prod_{y\in Y_S^\mathrm{cl}}\mathscr{F}_y^\mathrm{nr}\rlap{\ .}
\end{align*}
where we recall that
\begin{align*}
    \mathscr{F}_x^\mathrm{nr}&:=H^1(G_{K,x}^\mathrm{nr},(M_1^\vee)^{I_{K,x}})\times H^1(G_{K,x}^\mathrm{nr},\pi(M^{I_{K,x}}))\rlap{\ ,} \quad x \in X^{\mathrm{cl}},
\end{align*}
and the space of fields $\mathscr{F}(X_\emptyset)$ on $X=X_\emptyset$ is the pullback
\begin{align*}
\begin{aligned}
    \xymatrix{
    \mathscr{F}(X) =\mathscr{F}(X_\emptyset) \ar[d] \ar[r] & H^1(G_K,M_1^\vee)\times H^1(G_K,M_2) \ar[d] \\
    \mathrm{BC}(\partial X_\emptyset) \ar[r] & \displaystyle\prod_{x\in X^\mathrm{cl}}\mathscr{F}_x\rlap{\ .}
    }
\end{aligned}
\end{align*}

Let
\begin{eqnarray*}
    \mathscr{W}=\prod_{x \in X \setminus Y}\{0\}\times \prod_{y\in Y^{\mathrm{cl}}}\mathscr{W}_y,
    \quad \mathscr{W}_y =H^1(G_{K,y}^{\mathrm{nr}}, M^{I_{K,y}}), \quad y \in Y^{\mathrm{cl}}.
\end{eqnarray*}

\begin{lem} We have
    \begin{eqnarray*}
        \pi(\mathscr{W}_y) =H^1(G_{K,x}^\mathrm{nr},\pi(M^{I_{K,x}})), \quad y \in Y^{\mathrm{cl}}.
    \end{eqnarray*}
\end{lem}
\begin{proof}
We have an exact sequence of $G_{K,y}^{\mathrm{nr}}$-modules:
\begin{align*}
    \xymatrix{0 \ar[r] & M_1^{I_{K,y}} \ar[r] & M^{I_{K,y}} \ar[r] & \pi(M_2^{I_{K,y}}) \ar[r] & 0\rlap{\ .}}
\end{align*}
This induces an exact sequence:
\begin{align*}
    \xymatrixcolsep{0.8pc}\xymatrix{\cdots \ar[r] & \mathscr{W}_y=H^1(G_{K,y}^\mathrm{nr},M^{I_{K,y}}) \ar[r]^-\pi & H^1(G_{K,y}^\mathrm{nr},\pi(M^{I_{K,y}})) \ar[r] & H^2(G_{K,y}^\mathrm{nr},M_1^{I_{K,y}}) \ar[r] & \cdots\rlap{\ .}}
\end{align*}
Since $G_{K,y}^{\mathrm{nr}}\simeq \hat{\mathbb{Z}}$, its cohomological dimension is 1 (see \cite[p.173, Example]{NSW08}). Thus, $H^2(G_{K,y}^{\mathrm{nr}}, M_1^{K_{K,y}})=0$, since $M_1^{K_{K,y}}$ is torsion. This implies the desired result.
\end{proof}
But note that $\iota^{-1} (\mathscr{W}_y)$ is not necessarily the same as $H^1(G_{K,x}^\mathrm{nr},M_1^{I_{K,y}})$.
The sequence
$$0 \to M_1^{I_{K,y}} \to M^{I_{K,y}} \to M_2^{I_{K,y}}$$
is not necessarily surjective on the right. If it is surjective, i.e. $\pi(M^{I_{K,y}})=M_2^{I_{K,y}}$ (note that this is a sufficient condition for \eqref{FX-to-D1D1} in Lemma \ref{FX} to hold), then $\iota^{-1}(\mathscr{W}_y)$ becomes exactly $H^1(G_{K,y}^{\mathrm{nr}}, M_1^{I_{K,y}})$, and by Tate local duality, its orthogonal complement $\iota^{-1}(\mathscr{W}_y)^{\perp}$ is the unramified subgroup of the dual $H^1(G_{K,y}^{\mathrm{nr}}, (M_1^\vee)^{I_{K,y}})$.

If we assume that $\iota^{-1}(\mathscr{W}_y)^{\perp} = H^1(G_{K,y}^{\mathrm{nr}}, (M_1^\vee)^{I_{K,y}})$ for $y \in Y^{\mathrm{cl}},$ then we show that the Cassels-Tate pairing for finite Galois modules can be naturally interpreted as the arithmetic BF functional $\mathrm{BF}_X:\mathscr{F}(X) \to \mathbb{Q}/\mathbb{Z}$ using the decomposition formula {\it (2)} of Theorem \ref{Decomp-formula}.

\begin{prop} \label{ctpbf}
Assume that
\begin{eqnarray}
    \iota^{-1}(\mathscr{W}_y)^{\perp}\times \pi(\mathscr{W}_y)= \mathscr{F}_y^\mathrm{nr}
    \quad y \in Y^{\mathrm{cl}}.
\end{eqnarray}
Then $\mathrm{Sel} (M_2,\mathscr{W}_2) \times \mathrm{Sel} (M_1^\vee, \mathscr{W}_1^\perp)$ is identified with $\mathscr{F}(X)$ and we have
    \begin{eqnarray*}
        \mathrm{CTP}_E(\rho_1, \rho_2)= \mathrm{BF}_X(\rho_1, \rho_2).
    \end{eqnarray*}
\end{prop}
\begin{proof}
Given $\rho=(\rho_1,\rho_2)\in\mathscr{F}(X)$, choose a representative of its image under \eqref{FXS-defn}:
\begin{align*}
    \alpha=(\alpha_1,\alpha_2)\in Z^1(G_K,M_1^\vee)\times Z^1(G_K,M_2)
\end{align*}

By \eqref{BFnrx-lift}, for each $x \in Y^{\mathrm{cl}}$,
\begin{align*}
    d\varphi_x^\mathrm{nr}=\partial_x\alpha_{1}\cup d(\sigma\circ\partial_x\alpha_{2})\quad\textrm{for some}\quad\varphi_x^\mathrm{nr}\in\frac{C^2(G_{K,x}^\mathrm{nr},\mathcal{O}_x^{\mathrm{nr},\times})}{B^2(G_{K,x}^\mathrm{nr},\mathcal{O}_x^{\mathrm{nr},\times})}\rlap{\ .}
\end{align*}
For each $x \in X^{\mathrm{cl}}$, we have
\begin{eqnarray*}
    d\left( \partial_x\alpha_{1} \cup( \alpha_{x,2,M}-\sigma \circ \partial_x\alpha_{2} )\right) = \partial_x\alpha_{1}\cup d(\sigma\circ \partial_x\alpha_{2}).
\end{eqnarray*}
Thus, \eqref{local-HGnr-vanishing} implies that
\begin{eqnarray}\label{sectionunr}
    \partial_x\alpha_{1} \cup( \alpha_{x,2,M}-\sigma \circ \partial\alpha_{2}) - \varphi_x^{\mathrm{nr}} \in B^2(G_{K,x}, \mathscr{O}_x^{\mathrm{nr}, \times}) 
\end{eqnarray}
is a coboundary.
Note that $d\epsilon = \alpha_1\cup d f  $ and we choose $\sigma \in \mathscr{R}$ such that $f=\sigma \circ \alpha_2$ and choose $\varphi \in C^2(G_K, K^{\mathrm{sep},\times})/B^2(G_K, K^{\mathrm{sep},\times})$ (as in \eqref{FUS-BF-functional}) such that 
\begin{eqnarray}\label{globalchoice}
   d \varphi = \alpha_1 \cup d(\sigma \circ \alpha_2)=d\epsilon. 
\end{eqnarray}

Then we have the following equality in $\mathbb{Q}/\mathbb{Z}$:
\begin{eqnarray*}
    \mathrm{CTP}_E(\rho_1, \rho_2) &:=& \sum_{v \in X^{\mathrm{cl}}} \mathrm{inv}_v (\gamma_v)\\
   &=& \sum_{v \in X^{\mathrm{cl}}} \mathrm{inv}_v \left(\partial_x\epsilon- \partial_x\alpha_{1} \cup( \alpha_{x,2,M}-\sigma \circ \partial_x\alpha_{2} ) \right)\\
   &\stackrel{\eqref{globalchoice}}{=}& \sum_{v \in X^{\mathrm{cl}}} \mathrm{inv}_v \left(\partial_x\varphi- \partial_x\alpha_{1}\cup ( \alpha_{x,2,M} -\sigma \circ \partial_x\alpha_{2} )\right)\\
   &\stackrel{\eqref{sectionunr}}{=}&\sum_{x\in T}\mathrm{inv}_x(\partial_x\varphi-\varphi_x^{\mathrm{nr}})\quad \text{ for some finite set $T$ of places}\\
   &=&\mathrm{BF}_{X_T}(\rho)-\mathrm{BF}_T^\mathrm{nr}(\partial^\mathrm{nr}_T\rho)\rlap{\ }
   \stackrel{\mathrm{Theorem}\ \ref{Decomp-formula}, (2)}{=} \mathrm{BF}_X(\rho_1, \rho_2).
\end{eqnarray*}
\end{proof}

\begin{rmk}
The assignment in \eqref{gv}
    $$(\rho_1,\rho_2)=([\alpha_1], [\alpha_2]) \mapsto \xi(\rho_1,\rho_2):= \partial_x\alpha_{1} \cup ( \alpha_{x,2,M}-\sigma \circ \partial_x\alpha_{2} )  $$
    gives a section to the $\mathbb{Q}/\mathbb{Z}$-torsor $\varpi_x:\mathscr{L}_x \to \mathscr{F}_x.$, i.e. an element $\xi$ in $\Gamma(\mathscr{F}_x, \mathscr{L}_x)$. By using this and \eqref{scrLS-trivialization}, one can obtain a similar relationship between the Cassels-Tate pairing and the arithmetic BF functional $\mathrm{BF}_{X_S}^{\xi}$ in\eqref{secBF}, when $S \neq \emptyset$.
\end{rmk}

\begin{comment}
\subsection{Global arithmetic BF functional with boundary}
We deal with the case $S\neq \emptyset$: we have
\begin{align*}
    \mathrm{BC}(\partial X_S)=\prod_{x\in X^\mathrm{cl}}\mathrm{BC}(\partial X_S)_x:=\prod_{x\in X_S\setminus Y}\{0\}\times\prod_{x\in S}\mathrm{BC}_x\times\prod_{y\in Y_S^\mathrm{cl}}\mathscr{F}_y^\mathrm{nr}\rlap{\ ,}
\end{align*}
and the space of fields $\mathscr{F}(X_S)$ on $X_S$ is the pullback
\begin{align*}
\begin{aligned}
    \xymatrix{
    \mathscr{F}(X_S) \ar[d] \ar[r] & H^1(G_K,M_1^\vee)\times H^1(G_K,M_2) \ar[d] \\
    \mathrm{BC}(\partial X_S) \ar[r] & \displaystyle\prod_{x\in X^\mathrm{cl}}\mathscr{F}_x\rlap{\ .}
    }
\end{aligned}
\end{align*}
Let
\begin{eqnarray*}
    \mathscr{W}=\prod_{x \in X_S \setminus Y}\{0\}\times \prod_{x \in S} \mathscr{W}_x \times \prod_{y\in Y_S^{\mathrm{cl}}}H^1(G_{K,y}^{\mathrm{nr}}, M^{I_{K,y}}).
\end{eqnarray*}
be a compact subgroup of $\prod'_v H^1(G_{K,v}, M)$ in the theory of the Cassels-Tate pairing.
\begin{prop}
    Assume that
    \begin{eqnarray*}
    \iota^{-1}(\mathscr{W}_x)^{\perp}\times \pi(\mathscr{W}_x)= \mathrm{BC}(\partial X_S)_x, \quad x \in X^{\mathrm{cl}}.
\end{eqnarray*}
Then
\end{prop}
\begin{proof}
The assumption means that
    \begin{eqnarray*}
    \iota^{-1}(\mathscr{W}_y)^{\perp}\times \pi(\mathscr{W}_y)= \mathscr{F}_y^\mathrm{nr},
    \quad y \in Y_S^{\mathrm{cl}}, \quad \iota^{-1}(\mathscr{W}_x)^{\perp}\times \pi(\mathscr{W}_x)=\mathrm{BC}(\partial X_S)_x, \quad x \in S.
\end{eqnarray*}
\end{proof}
\end{comment}

\subsection{Formulas for partition functions}
 We provide a formula for the partition function of Definition \ref{partftn-X}.
 Recall that
\begin{align*}
    Z_X:=\sum_{\rho\in\mathscr{F}(X)}\exp\left(2\pi\sqrt{-1}\mathrm{BF}_X(\rho)\right)\in\mathbb{C}\rlap{\ .}
\end{align*}
 
\begin{prop}\label{onshell}
Under the same assumption as Proposition \ref{ctpbf}, we have
    \begin{align*}
    Z_X = | \pi (\mathrm{Sel}(M,\mathscr{W}))| \cdot |\mathrm{Sel} (M_1^\vee, \mathscr{W}_1^\perp) |.
    \end{align*}
\end{prop}

\begin{proof}
By \cite[Theorem 2.1]{MS24} and \cite[Proposition 3.4]{MS21}, $\rho_2=\pi(\rho)$ lies in $\pi (\mathrm{Sel}(M,\mathscr{W}))$ if and only if $\mathrm{CTP}_E(\rho_2, \rho_1)=0$ for all $\rho_1 \in \mathrm{Sel} (M_1^\vee, \mathscr{W}_1^\perp)$. By using this and Proposition \ref{ctpbf}, we compute
\begin{align*}
    Z_X&:=\sum_{\rho\in\mathscr{F}(X)}\exp\left(2\pi\sqrt{-1}\mathrm{BF}_X(\rho)\right)\\
        &=\sum_{(\rho_2,\rho_1) \in \mathrm{Sel}(M,\mathscr{W})\times \mathrm{Sel} (M_1^\vee, \mathscr{W}_1^\perp)}\exp\left(2\pi\sqrt{-1}\mathrm{CTP}_E(\rho_1,\rho_2)\right)\\
        &= \sum_{(\rho_2,\rho_1) \in \pi (\mathrm{Sel}(M,\mathscr{W}))\times \mathrm{Sel} (M_1^\vee, \mathscr{W}_1^\perp) }1 \\
        &\quad+\sum_{(\rho_2,\rho_1) \in \mathrm{Sel}(M,\mathscr{W})\setminus \pi (\mathrm{Sel}(M,\mathscr{W}))\times \mathrm{Sel} (M_1^\vee, \mathscr{W}_1^\perp)}\exp\left(2\pi\sqrt{-1}\mathrm{CTP}_E(\rho_1,\rho_2)\right)\\
        &= | \pi (\mathrm{Sel}(M,\mathscr{W}))| \cdot |\mathrm{Sel} (M_1^\vee, \mathscr{W}_1^\perp)|.
\end{align*}
    In the final step, we use the fact that a nontrivial character sum is zero.
\end{proof}

This proposition can be viewed as a generalization of \cite[Proposition 1.2]{CaKi22} and \cite[Proposition 2.3]{PP25}.

\appendix
\section{Cohomology of commutative group schemes}\label{fppf-comparison}

As promised at the beginning of Section \ref{ClassicalBF-SoF}, this appendix gives detailed explanation on the identifications between Galois cohomology and fppf cohomology over fields and discrete valuation rings. Although materials in this section are standard, we collect these for making the article more self-contained.

Let $B$ be a scheme. Since an \'etale covering is an fppf covering, the identity functor on $\mathrm{Sch}_{/B}$ gives a continuous functor between the big sites (cf. \cite[\href{https://stacks.math.columbia.edu/tag/0DDK}{Tag 0DDK}]{Stacks}):
\begin{align*}
    \xymatrix{\mathrm{Id}_{\mathrm{Sch}_{/B}}:(\mathrm{Sch}_{/B})_\mathrm{\acute{e}t} \ar[r] & (\mathrm{Sch}_{/B})_\mathrm{fppf}\rlap{\ .}}
\end{align*}
Given a commutative group scheme $\mathcal{G}$ over $B$, this gives a canonical homomorphism of abelian groups via the Leray spectral sequence:
\begin{align}\label{et-fppf-map}
    \xymatrix{H_\mathrm{\acute{e}t}^r(B,\mathcal{G}) \ar[r] & H_\mathrm{fppf}^r(B,\mathcal{G})\rlap{\ .}}
\end{align}
Although it is not used in this article, see \cite[\href{https://stacks.math.columbia.edu/tag/0757}{Tag 0757}]{Stacks} for the comparison between the small \'etale topos and the big \'etale topos.

\begin{rmk}\label{et-fppf-smqp} By \cite[Theorem III.3.9]{Mil80}, \eqref{et-fppf-map} becomes an isomorphism if $\mathcal{G}$ is smooth and quasi-projective over $B$.
\end{rmk}

\begin{ex} Let $\mathcal{G}$ be a commutative group scheme over an open subset $U\subseteq X$. By construction of compactly supported cohomology on $U$, we get a corresponding canonical homomorphism of abelian groups from \eqref{et-fppf-map}:
\begin{align}\label{et,c-fppf,c-map}
    \xymatrix{H_\mathrm{\acute{e}t,c}^r(U,\mathcal{G}) \ar[r] & H_\mathrm{fppf,c}^r(U,\mathcal{G})\rlap{\ .}}
\end{align}
By Remark \ref{et-fppf-smqp} and the five lemma, \eqref{et,c-fppf,c-map} becomes an isomorphism if $\mathcal{G}$ is smooth and quasi-projective over $U$. In particular, if $\mathcal{G}=\mathbb{G}_m$, then we have the canonical map becomes an isomorphism:
\begin{align*}
    \xymatrix{H_\mathrm{\acute{e}t,c}^r(U,\mathbb{G}_m) \ar[r]^-\sim & H_\mathrm{fppf,c}^r(U,\mathbb{G}_m)}
\end{align*}
which is compatible with \eqref{Global-H3-inv}.
\end{ex}

\begin{ex}\label{et-fppf-small} Let $\mathcal{G}$ be a commutative group scheme over an open subset $Y\subseteq X$. If $K$ is of characteristic $0$, then $\mathcal{G}_K$ is smooth and quasi-projective over $K$ by \cite[\href{https://stacks.math.columbia.edu/tag/047N}{Tag 047N}]{Stacks} and \cite[\href{https://stacks.math.columbia.edu/tag/0BF7}{Tag 0BF7}]{Stacks}. Hence in this case \eqref{et-fppf-map} and \eqref{et,c-fppf,c-map} become isomorphism for every sufficiently small open subset $U\subseteq Y$.
\end{ex}

\begin{lem}\label{localcomparison} Let $V$ be a scheme and $\mathcal{N}$ a commutative finite flat $V$-group scheme. The canonical map coming from \eqref{et-fppf-map}:
\begin{align*}
\xymatrix{H_\mathrm{\acute{e}t}^r(V,\mathcal{N}) \ar[r] & H_\mathrm{fppf}^r(V,\mathcal{N})}
\end{align*}
becomes an isomorphism in each of the following cases.
\begin{quote}
(1) $V=\Spec\Bbbk$ with $\Bbbk$ a field.\\
(2) $V=\Spec R$ with $R$ a discrete valuation ring.
\end{quote}
\end{lem}
\begin{proof} In each case, by \cite[Remark A.8]{Mil06}, there is an exact sequence of commutative $V$-group schemes:
\begin{align}\label{NABexact}
\xymatrix{0 \ar[r] & \mathcal{N} \ar[r] & \mathcal{A} \ar[r] & \mathcal{B} \ar[r] & 0}
\end{align}
where $\mathcal{A}$ and $\mathcal{B}$ are projective abelian schemes over $V$. By \cite[Theorem III.3.9]{Mil80}, the comparison maps for $\mathcal{A}$ and $\mathcal{B}$ are isomorphisms. Therefore, the assertion follows from applying the five lemma to the long exact sequence associated to \eqref{NABexact}.
\end{proof}

\begin{rmk}\label{et-grp} Let $\Bbbk$ be a field. By \cite[\href{https://stacks.math.columbia.edu/tag/0BNE}{Tag 0BNE}]{Stacks}, we have an equivalence:
\begin{align*}
    \xymatrix{
    {\left\{\begin{array}{c} \textrm{commutative} \\ \textrm{finite \'etale} \\ \textrm{$\Bbbk$-group schemes} \end{array}\right\}} \ar[r]^-\sim & {\left\{\begin{array}{c} \textrm{finite discrete} \\ \textrm{$G_\Bbbk$-modules} \end{array}\right\}} & \mathcal{N} \ar@{|->}[r] & \mathcal{N}(\Bbbk^\mathrm{sep})\rlap{\ .}
    }
\end{align*}
which transforms the Cartier dual to the $\Bbbk^{\mathrm{sep},\times}$-dual as in Notation \ref{notation-overall}. Combined with Lemma \ref{localcomparison}, we get canonical isomorphism:
\begin{align*}
    \xymatrix{H^r(G_\Bbbk,\mathcal{N}(\Bbbk^\mathrm{sep})) \ar[r]^-\sim & H_\mathrm{\acute{e}t}^r(\Bbbk,\mathcal{N}) \ar[r]^-\sim & H_\mathrm{fppf}^r(\Bbbk,\mathcal{N})}
\end{align*}
Moreover, if $\Bbbk$ is of characteristic $0$, then every finite $\Bbbk$-group scheme is \'etale by \cite[\href{https://stacks.math.columbia.edu/tag/047N}{Tag 047N}]{Stacks} so this works for every finite $\Bbbk$-group scheme.
\end{rmk}

Let $Y\subseteq X$ be an open subset. Let $\mathcal{M}$ be a nice commutative group scheme over $Y$. By Remark \ref{et-grp}, we have canonical isomorphisms:
\begin{align*}
    \xymatrix{H^r(G_K,\mathcal{M}(K^\mathrm{sep})) \ar[r]^-\sim & H_\mathrm{\acute{e}t}^r(K,\mathcal{M}) \ar[r]^-\sim & H_\mathrm{fppf}^r(K,\mathcal{M})\rlap{\ .}}
\end{align*}
For each $x\in X\setminus U$ we also have canonical isomorphisms:
\begin{align*}
    \xymatrix{H^r\left(G_{K,x},\mathcal{M}(K_x^\mathrm{sep})\right) \ar[r]^-\sim & H_\mathrm{\acute{e}t}^r(K_x,\mathcal{M}) \ar[r]^-\sim & H_\mathrm{fppf}^r(K_x,\mathcal{M})\rlap{\ .}}
\end{align*}
By \cite[\href{https://stacks.math.columbia.edu/tag/01TG}{Tag 01TG}]{Stacks}, which applies because $\mathcal{M}_K\rightarrow\Spec K$ is finite, the restriction
\begin{align*}
    \xymatrix{\mathcal{M}(K^\mathrm{sep}) \ar[r] & \mathcal{M}(K_x^\mathrm{sep})}
\end{align*}
is bijective. By functoriality, this is a $G_{K,x}$-linear isomorphism. Combined with the above isomorphisms, we get a canonical isomorphism:
\begin{align*}
    \xymatrix{H^r(G_{K,x},\mathcal{M}(K^\mathrm{sep})) \ar[r]^-\sim & H_\mathrm{fppf}^r(K_x,\mathcal{M})\rlap{\ .}}
\end{align*}

\begin{rmk}\label{Hfppftop} By \cite[1.11]{Kes17}, Remark \ref{et-grp} and \cite[Theorem 7.1.8]{NSW08}, each $H_\mathrm{fppf}^1(K_x,\mathcal{M})$ is finite discrete. If $y\in Y^\mathrm{cl}$, then the pullback along $\mathcal{O}_y\hookrightarrow K_y$ defines an injection
\begin{align*}
    \xymatrix{H_\mathrm{fppf}^1(\mathcal{O}_y,\mathcal{M}) \ar@{^(->}[r] & H_\mathrm{fppf}^1(K_y,\mathcal{M})}
\end{align*}
by \cite[1.11]{Kes17}. See \cite[section 3]{Kes15} for details.
\end{rmk}

By \cite[Remark III.2.21]{Mil80}, we have a first quadrant spectral sequence for each $y\in Y^\mathrm{cl}$:
\begin{align*}
E_2^{p,q}=H^p(G_{K,y}^\mathrm{nr}, H^q(\mathcal{O}_y^\mathrm{nr}, \mathcal{M})) \Rightarrow H_\mathrm{\acute{e}t}^{p+q}(\mathcal{O}_y,\mathcal{M})\rlap{\ .}
\end{align*}
By \cite[\href{https://stacks.math.columbia.edu/tag/03QO}{Tag 03QO}]{Stacks}, $H^q(\mathcal{O}_y^\mathrm{nr},\mathcal{M})=0$ for $q>0$ because $\mathcal{O}_y^\mathrm{nr}$ is a strictly Henselian local ring. Then the above spectral sequence degenerates and hence the edge maps become isomorphisms:
\begin{align*}
    \xymatrix{H^r(G_{K,y}^\mathrm{nr},\mathcal{M}(\mathcal{O}_y^\mathrm{nr})) \ar[r]^-\sim & H_\mathrm{\acute{e}t}^r(\mathcal{O}_y,\mathcal{M})}
\end{align*}
By functoriality and Lemma \ref{localcomparison}, we have the following commutative diagram:
\begin{align}\label{unram-coh}
\begin{aligned}
    \xymatrix{
    H^r(G_{K,y}^\mathrm{nr},\mathcal{M}(\mathcal{O}_y^\mathrm{nr})) \ar[d] \ar[r]^-\sim & H_\mathrm{\acute{e}t}^r(\mathcal{O}_y,\mathcal{M}) \ar[d]^-\wr \ar[r]^-\sim & H_\mathrm{fppf}^r(\mathcal{O}_y,\mathcal{M}) \\
    H^r\left(G_{\kappa(y)},\mathcal{M}(\kappa(y)^\mathrm{sep})\right) \ar[r]^-\sim & H_\mathrm{\acute{e}t}^r(\kappa(y),\mathcal{M}) &
    }
\end{aligned}
\end{align}
where the middle column is an isomorphism by applying \cite[\href{https://stacks.math.columbia.edu/tag/09ZI}{Tag 09ZI}]{Stacks} to the complete discrete valuation ring $\mathcal{O}_y$. Hence the left column becomes an isomorphism. On the other hand, by the functoriality of $\mathcal{M}$, we have the following commutative diagram of $G_{K,y}$-modules:
\begin{align}\label{unram-pts}
\begin{aligned}
    \xymatrix{
    \mathcal{M}(\mathcal{O}_y^\mathrm{nr}) \ar[d]^-\wr \ar[r]^-\sim & \mathcal{M}(\mathcal{O}_y^\mathrm{sep})^{I_{K,y}} \ar[d]^-\wr \ar@{^(->}[r] & \mathcal{M}(\mathcal{O}_y^\mathrm{sep}) \ar[d]^-\wr & \\
    \mathcal{M}(K_y^\mathrm{nr}) \ar[r]^-\sim & \mathcal{M}(K_y^\mathrm{sep})^{I_{K,y}} \ar@{^(->}[r] & \mathcal{M}(K_y^\mathrm{sep}) & \mathcal{M}(K^\mathrm{sep}) \ar[l]_-\sim
    }
\end{aligned}
\end{align}
where the columns are isomorhisms by the valuative criterion \cite[\href{https://stacks.math.columbia.edu/tag/0BX5}{Tag 0BX5}]{Stacks}. Combining this with the restriction along the quotient $G_{K,y}\rightarrow G_{K,y}^\mathrm{nr}$, we get an injection of complexes:
\begin{align}\label{unram-subcomplex}
    \xymatrix{C^\bullet(G_{K,y}^\mathrm{nr},\mathcal{M}(\mathcal{O}_y^\mathrm{nr})) \ar@{^(->}[r] & C^\bullet(G_{K,y},\mathcal{M}(K^\mathrm{sep}))}
\end{align}
In particular, the map on $H^1$ induced from \eqref{unram-subcomplex} fits into the inflation-restriction exact sequence \cite[Corollary 2.4.2]{NSW08} and gives an isomorphism:
\begin{align*}
    H^1(G_{K,y}^\mathrm{nr},\mathcal{M}(\mathcal{O}_y^\mathrm{nr}))\cong\ker\left(\xymatrixcolsep{1.25pc}\xymatrix{H^1(G_{K,y},\mathcal{M}(K^\mathrm{sep})) \ar[r] & H^1(I_{K,y},\mathcal{M}(K^\mathrm{sep}))}\right)
\end{align*}
where the right hand side is usually called the subgroup of \emph{unramified classes}. On the other hand, applying \eqref{unram-pts} to $\mathcal{M}^\vee$, we get canonical identifications:
\begin{align*}
    (\mathcal{M}(K_y^\mathrm{sep})^\vee)^{I_{K,y}}\cong\mathcal{M}^\vee(K_y^\mathrm{sep})^{I_{K_y}}\cong\mathcal{M}^\vee(\mathcal{O}_y^\mathrm{nr})
\end{align*}
where the first isomorphism comes from Remark \ref{et-grp}. Hence the middle term of \ref{Fxnr-to-Fx} is canonically identified with the fppf cohomology.

\end{document}